\documentclass[11pt]{amsart}

\usepackage{amssymb,amsfonts}
\usepackage{graphicx}
\usepackage[all,arc]{xy}
\usepackage{enumerate}
\usepackage{mathrsfs}
\usepackage{tikz-cd}
\usepackage{float}
\usepackage[T1]{fontenc}
\usetikzlibrary{calc}

\usetikzlibrary{decorations.pathreplacing}

\usepackage{mathtools}
\usepackage{ytableau}
\usepackage[left=1in,top=1in,right=1in,bottom=1in]{geometry}

\newtheorem{thm}{Theorem}[section]
\newtheorem*{thm*}{Theorem}

\newtheorem{cor}[thm]{Corollary}
\newtheorem{prop}[thm]{Proposition}
\newtheorem{lem}[thm]{Lemma}

\newtheorem{conj}[thm]{Conjecture}

\theoremstyle{definition}
\newtheorem{defn}[thm]{Definition}

\newtheorem{con}[thm]{Construction}
\newtheorem{exmp}[thm]{Example}

\newtheorem{rem}[thm]{Remark}
\newtheorem{ques}[thm]
{Question}

\theoremstyle{remark}

\newcommand{\Hom}{\mathrm{Hom}}

\newcommand{\Spec}{\text{Spec}\,}

\renewcommand{\emptyset}{\varnothing}

\makeatletter
\let\c@equation\c@thm
\makeatother
\numberwithin{equation}{section}

\usepackage{hyperref}

\title{Chow Vanishing and Motives of Cluster Varieties}
\author{Josephine Hlavinka}
\address{Josephine Hlavinka, Department of Mathematics, University of California, Berkeley,
         970 Evans Hall, Berkeley, CA 94720-3840, USA}
\email{josephhlavinka@berkeley.edu}

\begin{document}
    \begin{abstract}
    We prove that the integral Chow groups $CH_i$ and mixed Hodge degree $H^{2i, (i, i)}$ cohomology groups of a really full rank (RFR) sink-recurrent cluster variety vanish for $i > 0$. In particular this applies to braid varieties and open projected Richardson varieties in any Lie type. Our main technical tool is the construction of a stratification of any RFR sink-recurrent cluster variety $\mathcal{A}(\Sigma)$ over any field $k$ into (affine space times) RFR sink-recurrent cluster varieties of seeds with fewer mutable vertices than $\Sigma$, with which we show that RFR sink-recurrent cluster varieties are linear in the sense of Totaro. We also show that cycle class maps are isomorphisms onto the lowest-weight part of rational Borel-Moore homology for any mixed Tate variety over a number field. Finally, we use our results to deduce vanishing theorems for the Khovanov-Rozansky homology groups of closures of positive braids and generation properties for the Chow groups of flag Richardson, projected Richardson, and brick varieties: namely, that these groups are generated integrally by the fundamental classes of their closed Richardson, positroid, and braid strata, respectively.
\end{abstract}
\maketitle
\tableofcontents
\bigskip
\section{Introduction}

\subsection{Background and Results}

\textit{Background.} Let $\mathcal{A}(\Sigma)$ denote the cluster variety of a skew-symmetrizable seed $\Sigma$. A crucial development in cluster geometry over the last decade has been the realization that many cluster varieties which arise in practice -- including all braid varieties and open projected Richardson varieties (see \cite{MR4868947} and \cite{galashin2026braid}) -- are of the form $\mathcal{A}(\Sigma)$ for a \textit{really full rank} (RFR) and \textit{sink-recurrent} seed $\Sigma$. The import of this realization comes from the fact that in the same span of time, strong constraints have been discovered on the cohomology of any complex RFR sink-recurrent $\mathcal{A}(\Sigma)$. To start, all such cluster varieties are smooth by the main theorem of \cite{muller2013locally}. But more interesting things are true: for example, the mixed Hodge structures on the cohomology of $\mathcal{A}(\Sigma)$ are mixed Tate, i.e. $H^{i, (p, q)}(\mathcal{A}(\Sigma)) = 0$ if $p \neq q$; and if $\mathcal{A}(\Sigma)$ is even-dimensional then $H^*(\mathcal{A}(\Sigma))$ enjoys a ``curious" hard Lefschetz theorem with respect to certain classes $[\gamma] \in H^{2, (2, 2)}(\mathcal{A}(\Sigma))$. These results are due to \cite{lam2022cohomology} under the assumption that $\Sigma$ is Louise and skew-symmetric and to \cite{galashin2026braid} in general.\\

In turn, many RFR sink-recurrent cluster varieties have been shown to admit mixed Hodge structures which determine other interesting invariants. For example, the torus equivariant mixed Hodge groups of a braid variety $X(\beta \Delta)$ compute the $a = 0$ part of the Khovanov-Rozansky homology of the $(-1)$-framed closure of $\beta$ (see \cite{MR5110495}); certain mixed Hodge groups of open Richardson varieties compute certain graded Ext groups between Verma modules (see \cite{MR4809331}); and the mixed Hodge Poincar\'e polynomials of certain open positroid varieties compute $(q, t)$-Catalan numbers (see \cite{MR4809331}).\\

An emerging feature in this study is that complex RFR sink-recurrent cluster varieties often enjoy certain cohomological vanishing properties, and that these vanishing theorems can be used to actually compute some of the above invariants. This technique appears in the proof of Theorem 1.6 of \cite{barkley2026combinatorial}, which shows that for Verma modules $M_u$, $M_v$ in any Lie type, the groups \begin{equation}
    \text{Ext}^{\ell(v) - \ell(u) - 1}(M_u, M_v)
\end{equation} are a combinatorial invariant of the interval $[u, v]$ in the Bruhat order on the corresponding Weyl group. Their proof comes down to the observation that the mixed Hodge groups $H^{1, (0, 0)}(R_{u,v}^\circ)$ and $H^{2, (1, 1)}(R_{u,v}^\circ)$ of the open Richardson variety $R_{u,v}^\circ$ associated to $[u, v]$ are both $0$.\\

Two particular complements to their result are worth mentioning. One is Theorem 4.13 of \cite{cao2024valuation}, which shows that all primitive full rank (and in particular, all RFR) locally acyclic cluster varieties over a characteristic $0$ field have vanishing Picard groups: $CH^1(\mathcal{A}(\Sigma)) = 0$. Another is Corollary 1.4 of \cite{LSII}, which says (after an application of curious Poincar\'e symmetry) that if $\Sigma$ is skew-symmetric and acyclic then \cite{barkley2026combinatorial}'s vanishing holds in \textit{every} degree. That is, for an acyclic really full rank quiver $Q$ and over the complex numbers, we have that \begin{equation}
H^{2i, (i, i)}(\mathcal{A}(Q)) = 0
\end{equation}
for all $i > 0$.

\textit{Results.} The main computational result of this paper is that the same Chow and cohomological vanishing theorems hold in all degrees, for all really full rank sink-recurrent cluster varieties, and (in the Chow case) over any field.

\begin{thm}[Theorem \ref{chow/cohom theorem}]
    Let $\mathcal{A}(\Sigma)$ be a $d$-dimensional RFR sink-recurrent cluster variety over a field $k$. Then the homological Chow groups $CH_i(\mathcal{A}(\Sigma))$ vanish in all degrees $i < d$: 
     \begin{equation}
        CH_i(\mathcal{A}(\Sigma)) = \begin{cases}  0 & i < d \\
    \mathbb{Z} & i = d.
   \end{cases} \end{equation} By Theorem \ref{cycle class iso}, over $\mathbb{C}$ this implies that $H^{2i, (i, i)}(\mathcal{A}(\Sigma)) = 0$ for $i > 0$, and curious Poincar\'e symmetry then implies that
    $H^d(\mathcal{A}(\Sigma)) = H^{d, (d, d)}(\mathcal{A}(\Sigma)) \cong \mathbb{Q}$.
\end{thm}

In particular, over $\mathbb{C}$ we learn that the restriction map from $H^d(\mathcal{A}(\Sigma))$ to $H^d((\mathbb{C}^*)^d)$ induced by the inclusion of any cluster torus is an isomorphism. This generalizes a result of \cite{LSII} from the acyclic case (see Remark \ref{top cohom remark}).\\

Let us explain the clauses of this theorem in order. We prove the Chow vanishing by constructing a stratification of any RFR sink-recurrent cluster variety $\mathcal{A}(\Sigma)$ into pieces of the form \begin{equation}
    \mathcal{A}(\Sigma) \cong \mathcal{A}(\Sigma^{\backslash s}) \sqcup (\mathbb{A}^1 \times \mathcal{A}(\Sigma^{/s})),
\end{equation}
where $\Sigma^{\backslash s}$ and $\Sigma^{/s}$ are certain seeds which have fewer mutable vertices than $\Sigma$ (see Subsection \ref{subsection 4.2}). A similar decomposition was considered in Section 3 of \cite{galashin2026braid} for the sake of \textit{constructing} really full rank sink-recurrent cluster structures on varieties. The theoretical crux of our work is a converse: that any RFR sink-recurrent cluster variety admits such a decomposition (see Proposition \ref{vanishing locus is linear}), and that the pieces of this decomposition are still (affine space times) RFR sink-recurrent cluster varieties (see Proposition \ref{still rfr}). A version of this was proven for leaf-recurrent \textit{quivers} in Proposition 5.13 of \cite{galashin2024plabic}, and in the acyclic case this stratification is the same as the anticlique stratification considered in \cite{LSII}.\\

Now let us unravel the second clause and properly explain Theorem \ref{cycle class iso}. Curious Poincar\'e symmetry implies that \begin{equation}
    H^{2i, (i, i)}(\mathcal{A}(\Sigma)) \cong H^{d, (d-i, d-i)}(\mathcal{A}(\Sigma)),
\end{equation}
so this statement would follow from the first if the cycle class maps from $CH^i(\mathcal{A}(\Sigma)) \otimes \mathbb{Q}$ to $H^{2i, (i, i)}(\mathcal{A}(\Sigma))$ were isomorphisms. To establish this we employ the formalism of mixed Tate (Voevodsky) motives, recalled in Section \ref{section 2}. A variety $X$ having a mixed Tate motive implies, but is strictly stronger than, $X$ having mixed Tate cohomology.\\

Our main motivic theorem provides exactly the kind of tool we seek: we prove that the rational cycle class map from $CH_i(X) \otimes \mathbb{Q}$ to bottom weight Borel-Moore homology is an isomorphism for mixed Tate $X$ over a number field.

\begin{thm}[Theorem \ref{cycle class iso}]
For any inclusion $\sigma: k \hookrightarrow \mathbb{C}$ of a number field into the complex numbers and mixed Tate variety $X/k$, the cycle class map $cl_i: A_i(X) \rightarrow gr_{-2i}H_{2i}^{BM}(X(\mathbb{C});\mathbb{Q})$ induced by the corresponding Hodge realization is an isomorphism.
\end{thm}

We then show in Corollary \ref{sink-recurrent MT} that any RFR sink-recurrent cluster variety is mixed Tate (and in fact linear in the sense of \cite{totaro2014chow}, see Theorem \ref{RFR LINEAR THM}), completing the proof of the second clause of Theorem \ref{chow/cohom theorem}. Using this we show in Theorem \ref{sink-recurrent split} that such varieties have motives which split over $\mathbb{Q}$ into direct sums of Tate twists. In totality this provides a motivic lift of the main results of \cite{lam2022cohomology}, which show that RFR, Louise, and skew-symmetric cluster varieties have mixed Tate cohomology and $\mathbb{Q}$-split weight filtrations. It also weakens their hypotheses, dropping skew-symmetry and replacing Louiseness with sink-recurrence (see Proposition \ref{sink-recurrent CHL} for a cohomological version of this).\\

In the process we compare mixed Tateness to other notions which are currently prominent in the broader motivic literature, namely \cite{totaro2014chow}'s more restrictive notion of linearity (see Subsection \ref{subsection 3.1}) and \cite{MR5083265}'s Chow-K\"unneth generation property (CKgP, see Subsection \ref{subsection 3.2}). We deduce 1) that Totaro's main computational formulae for linear schemes also apply to mixed Tate ones, and 2) that a smooth and proper variety is mixed Tate if and only if it has the CKgP (see Theorem \ref{mixed Tate vs. ckgp}). We contend that Section 3 exhibits mixed Tateness as the right ``Tate-type" notion to consider for proper singular varieties, and we expect its contents to be of independent interest.\\

We end with a number of applications. In Subsection \ref{subsection 6.1} we use the previous results to compute certain triply-graded Khovanov-Rozansky homology groups of positive (type A) braids $\beta \in Br_n^+$. Using \cite{MR5110495}'s identification of the groups $HHH^{a = 0, t, q}(\beta)$ with mixed Hodge pieces of torus equivariant compactly supported cohomology groups of the braid variety $X(\beta \Delta)$ we learn the following.

\begin{thm}[Theorem \ref{HHH vanishing}]
    For all $q = -t$ and any positive braid $\beta$, the triply-graded Khovanov-Rozansky homology group $HHH^{0, t, -t}(\beta)$ satisfies \begin{equation}
        HHH^{0, t, -t}(\beta) = \begin{cases}  0 & t \neq \ell(\beta) \\
    \mathbb{Q} & t = \ell(\beta) .
   \end{cases} \end{equation}
\end{thm}

This generalizes a portion of the main theorems of \cite{li2025last}, which cover the case where $t = -q$ is either $\ell(\beta)$ or $\ell(\beta) -2$, to the entire $t = -q$ diagonal. We also learn that the Chow and cohomology groups of an augmentation variety $Aug(\beta, \mathfrak{t}_c)$ vanish in the same range as those of the corresponding braid variety (see Theorem \ref{augmentation vanishing}). In Corollary \ref{forest vanishing} we show further that an analogue of Theorem \ref{HHH vanishing} also holds in the $t = -q + 1$ case if $X(\beta \Delta)$ is $\mathcal{A}(Q)$ for an RFR acyclic or forest quiver $Q$, and conjecturally if $\beta$ is any positroid braid (see Conjecture \ref{positroid braid conj}). This complements the knot-homological analysis of forest quivers in \cite{schwartz2026homfly}.\\

Finally, in Subsection \ref{subsection 6.2} we deduce that the Chow groups of a projected Richardson variety or brick variety are generated by the fundamental classes of its (Richardson, positroid, or braid) strata. The main theorem of \cite{levinson2022class} is essentially the codimension $1$ case for Richardson varieties in Grassmannians (with respect to the usual Richardson stratification). In exchange for considering \textit{projected} Richardsons we are able to get the corresponding result in all degrees and Lie types.

\begin{thm}[Theorem \ref{generation}]
     Let $X$ be a flag Richardson variety, projected Richardson variety, or brick variety (resp.) over a field, and let $\{S_i\}$ denote the set of closed strata of its Richardson, positroid, or braid stratification (resp.). Then $CH_*(X)$ is generated by the classes $[S_i]$, and the same is true of the lowest weight rational Borel-Moore homology groups of $X(\mathbb{C})$.
\end{thm} We end by raising the question of what combinatorial or geometric objects the Chow groups of such an $X$ parametrize, and (having shown through the preceding theorem that the latter is well-posed) sharpen this to ask in particular for combinatorially interesting collections of strata whose fundamental classes form a basis of $CH_*(X)$.

\subsection{Outline, Conventions, and AI Disclosure} \textit{Outline.}
In Section \ref{section 2} we recall the formalism of Voevodsky motives in sufficient detail to apply it to our setting. In Section \ref{section 3} we use the properties of mixed Tate motives deduced therein to study general comparisons between mixed Tateness, linearity, and the CKgP. In Section \ref{section 4} we turn to cluster theory, exhibiting a stratification of RFR sink-recurrent cluster varieties into (affine spaces times) smaller ones. We then use this stratification to establish mixed Tateness and splitness of said varieties, as well as mixed Tateness of compactifications thereof. In Section \ref{section 5} we deduce homological and Chow-theoretic consequences, establishing our vanishing theorem. In Section \ref{section 6} we end by computing knot homology groups and deducing generation results for compactifications of RFR sink-recurrent cluster varieties.\\

\textit{Conventions.} $CH_i(X)$ denotes the $i$-th integral homological Chow group of a scheme $X$, and $CH^i(X)$ its cohomological analogue. For most of the paper we will use homological Chow indexing, matching the conventions of \cite{totaro2014chow}. All schemes are separated and finite type over a base field $k$, and a variety is a reduced and irreducible scheme. Given a scheme $X$, we will write $H^0(X)$ to refer to its global sections. Also, given $x \in H^0(X)$ we will write $D(x)$ for the locus on which $x$ is invertible and $V(x)$ for its vanishing locus.\\ 

Unless otherwise indicated cohomology is always taken with rational coefficients. Every complex variety we consider will have a mixed Hodge structure with weight filtration defined over $\mathbb{Q}$. Therefore, given a complex variety $X$, we will write $H^{i, (p, q)}(X)$ for the $(p, q)$ part of the $i$-th (rational) cohomology of $X$, i.e. for the $(p, q)$ part of $gr_{p+q}^W H^i(X)$. We will write $h^{k,(p, q)}(X)$ for the dimension of $H^{k, (p, q)}(X)$, and at times even write $h^k(X)$ for the dimension of $H^k(X)$. We refer to the cohomology groups $H^{2i, (i, i)}(X)$ as the \textit{pure cohomology} of $X$. Finally, we use the name (open) positroid variety to refer to projections of (open) Richardson varieties from the flag variety to a parabolic quotient in arbitrary Lie type. We will never discuss Richardson varieties in a non-complete flag variety except as cases of projected Richardson varieties, and so we will refer to Richardson varieties in $G/B$ as \textit{flag} Richardson varieties.\\

\textit{AI Disclosure.}
An LLM, which we decline to name for reasons akin to those given in Footnote 3 of \cite{kelly2026some}, was used to check proofs, copy edit prose, and ensure accuracy of references in final drafts of the manuscript. Further details are available upon request. All ideas and arguments are due to the author, and any remaining errors are hers.

\subsection{Acknowledgments} We thank Gabriel Beiner, Eugene Gorsky, Soyeon Kim, Thomas Lam, David Nadler, David Speyer, and Peng Zhou for helpful discussions. We are also grateful to Cole Franklin, Charlotte Kiesling, and Sophie McCormick for their close reading of an earlier draft of this work. Particular thanks are due to Roman Krutowski, whose questions led to the development of this project, and to Sonia Choy for her support, interest, and cluster theoretic wisdom. The author's research is partially supported by the National Science Foundation Graduate Research Fellowship
Program under Grant No. DGE 2146752 and by a Simons Dissertation Fellowship.

\section{Background on Voevodsky Motives}\label{section 2}

In this section we fix a field $k$ (which will be a number field unless otherwise stated) and study the triangulated category of rational Voevodsky motives $DM(k) := DM(k; \mathbb{Q})$ over $k$. We refer the reader to \cite{dupont2024introduction} for a wonderful and much deeper survey on $DM(k)$. This section contains no original content, and we learned everything here from either \cite{deligne2005groupes}, \cite{dupont2024introduction}, or \cite{totaro2016motive}.\\

Every separated scheme $X$ of finite type over $k$ determines an object $M^c(X)$ in $DM(k)$ called the \textit{compactly supported motive} of $X$. This assignment is covariantly functorial in proper morphisms, and any closed immersion $Z \hookrightarrow X$ of separated finite type schemes over $k$ induces a triangle in $DM(k)$ of the form \begin{equation}
    M^c(Z) \rightarrow M^c(X) \rightarrow M^c(X \setminus Z) \rightarrow M^c(Z)[1],
\end{equation} where $[1]$ is the shift functor in $DM(k)$. We refer to this triangle as the \textit{excision sequence} resulting from $Z \hookrightarrow X$.\\

A salient feature of $DM(k)$ is that it contains \textit{Tate twists} $\mathbb{Q}(i)$, which can be defined by \begin{equation}
    \mathbb{Q}(i) := M^c(\mathbb{A}^i)[-2i]
\end{equation} \noindent for $i \geq 0$, and $\mathbb{Q}(-i) = \underline{\Hom}_{DM(k)}(\mathbb{Q}(i), \mathbb{Q})$ in general. To make the latter sensible we note that $DM(k)$ is in fact \textit{tensor} triangulated and has an internal Hom, structures which we denote by $\otimes$ and $\underline{\Hom}$ respectively. A key property of $\otimes$ is that $\mathbb{Q}$ is the monoidal unit and \begin{equation}
    \mathbb{Q}(i) \otimes \mathbb{Q}(j) \cong \mathbb{Q}(i + j).
\end{equation}
Much of $DM(k)$'s use lies in studying rational Chow groups of varieties. Indeed, writing $A_i(X) := CH_i(X) \otimes \mathbb{Q}$, we have \begin{equation}
    A_i(X) \cong \Hom(\mathbb{Q}(i)[2i], M^c(X)).
\end{equation}
For a number field $k$ and $n \geq 1$, Levine proved (see  \cite{levine1998mixed} and Section 6.3 of \cite{dupont2024introduction}) that \begin{equation}\label{MT hereditary}
    DM^i(k)(\mathbb{Q}(-n), \mathbb{Q}) =  \begin{cases} 
      K_{2n-1}(k) \otimes \mathbb{Q} & i = 1 \\
     0 & i \neq 1.
   \end{cases}
\end{equation}
Further, $DM(k)(\mathbb{Q}(n), \mathbb{Q})$ vanishes for $n \geq 1$, and one has \begin{equation}\label{MT}
       \Hom(\mathbb{Q}(i), \mathbb{Q}(i)) = \mathbb{Q}. 
\end{equation}
   That Tate twists have no negative Exts between them implies that one can define a $t$-structure on $DMT(k)$, the thick triangulated closure of the set of Tate motives $\mathbb{Q}(i)$ inside of $DM(k)$.
   
   \begin{defn}
       The abelian category $MT(k)$ of (rational) \textit{mixed Tate motives} over a number field $k$ is the heart of the Levine $t$-structure on $DMT(k)$, i.e. the closure of the Tate twists in $DMT(k)$ under iterated extensions. $DMT(k)$ is its derived category, the \textit{triangulated category of mixed Tate motives} over $k$. A motive $M$ will be called mixed Tate if $M \in DMT(k)$.
   \end{defn}

   \begin{rem}
   We will use in passing later that there is in fact a category $DM(k)$ of Voevodsky motives over any field $k$. Further, $DMT(k)$ is \textit{also} sensible over any field $k$. Many of the properties we allude to here which hold when $k$ is a number field are conjectural in this setting, however.
   \end{rem}

$MT(k)$ satisfies an assortment of pleasant properties. First, every object $M$ of $MT(k)$ admits a finite increasing weight filtration $W$ indexed by \textit{even} integers, such that the weight filtration on $\mathbb{Q}(-n)$ is supported only in degree $2n$ and (more generally) that $gr_{2n}^W M$ is a direct sum of $\mathbb{Q}(-n)$'s. Morphisms in $MT(k)$ respect the weight filtration, and if $X$ is a mixed Tate variety over $k$ then the weight filtration on the $i$-th cohomology object of $M^c(X)$ with respect to Levine's $t$-structure is supported only in degrees in the interval $[-i, 0]$. This $i$-th cohomology object is denoted by $H_i^{M,BM}(X)$ and referred to as the $i$-th \textit{compactly supported motivic homology} of $X$.

\begin{defn}\label{Mixed Tate}
    A variety $X$ over a field $k$ is called \textit{mixed Tate} if $M^c(X) \in DMT(k)$. When $k$ is a number field, a mixed Tate motive $M$ is called \textit{split} if it decomposes as a direct sum of shifts of Tate twists $\mathbb{Q}(j)$, and a mixed Tate variety $X$ over a number field $k$ is said to be split if each cohomology object $H_i^{M, BM}(X)$ is.
\end{defn}

Given any field extension $k \subset k'$ there is an exact, symmetric monoidal, and weight filtration preserving base change functor from $DM(k)$ to $DM(k')$ sending $\mathbb{Q}(i)$ to $\mathbb{Q}(i)$ and $M^c(X_k)$ to  $M^c(X_{k'})$. It therefore sends $DMT(k)$ to $DMT(k')$ and split motives to split motives.\\

For any embedding $\sigma: k \hookrightarrow \mathbb{C}$, there is an exact functor $R^H_{k, \sigma}$ from $MT(k)$ to $MHS_k$, the (hereditary) abelian category of rational mixed Hodge structures with Hodge filtrations defined over $k$, which sends $\mathbb{Q}(-n)$ to $\mathbb{Q}(-n)$ and induces the identity map from $End_{MT(k)}(\mathbb{Q}(i))$ to $End_{MHS_k}(\mathbb{Q}(i)) \cong \mathbb{Q}$, and which sends $H_i^{M, BM}(X)$ to $H_i^{BM}(X(\mathbb{C}); \mathbb{Q})$ equipped with its weight and Hodge filtrations. Therefore (using heredity of both categories) one sees that the canonical extension of $R^H_{k, \sigma}$ to derived categories on each side sends $M^c(X)$ to $H_*^{BM}(X(\mathbb{C}); \mathbb{Q})$. Further, $R^H_{k,\sigma}$ preserves weight filtrations. This leads us to the following useful

\begin{defn}
    A variety over a number field $k$ equipped with an embedding $\sigma: k \hookrightarrow \mathbb{C}$ is said to have \textit{mixed Tate cohomology} if $H_{i, (p, q)}^{BM}(X(\mathbb{C})) = 0$ (equivalently when $X$ is smooth, if $H^{i, (p, q)}(X(\mathbb{C})) = 0$) for all $p \neq q$. This condition is independent of $\sigma$.
\end{defn}

Therefore we see that if $X$ is mixed Tate then the weight filtration on $H_i^{M, BM}(X)$ is supported only in \textit{even} degrees between $-i$ and $0$. In this case $X$ also has mixed Tate cohomology.\\

If $k = \mathbb{Q}$, the realization functor $R^H := R^H_{\mathbb{Q}, \sigma}$ from $MT(\mathbb{Q})$ to  $MHS_\mathbb{Q}$ corresponding to the unique embedding $\sigma: \mathbb{Q} \hookrightarrow \mathbb{C}$ is in fact faithful. It also induces injections on $\text{Ext}^1$'s. No matter the choice of number field $k$, however, the morphism \begin{equation}
    R_{k, \sigma}^H: \Hom(\mathbb{Q}(j)[i], M^c(X)) \rightarrow \Hom(\mathbb{Q}(j)[i], H_*^{BM}(X(\mathbb{C}); \mathbb{Q})) 
\end{equation} induced by any $R^H_{k, \sigma}$ identifies with the cycle class map from the higher Chow group $H_i^{M, BM}(X; \mathbb{Q}(j))$ to Beilinson's absolute Hodge cohomology. Whenever the source and target are identifiable under more elementary names, this cycle class map agrees with the expected one there.

\begin{lem}\label{2-of-3 mixed tate}
   Assume given a closed immersion of varieties $Z \hookrightarrow X$ over $k$ with complement $U := X \setminus Z$. Then if any two of $Z, X$ and $U$ are mixed Tate, the third also is.
\end{lem}

\begin{proof}
    By the excision sequence, the compactly supported motive of any one of these spaces lives in a triangle in $DM(k)$ with (a shift of) the other two. Since $DMT(k)$ is a thick triangulated subcategory of $DM(k)$ the result follows.
\end{proof}

Standard arguments then show that a variety which is stratified by mixed Tate ones, or which is covered by mixed Tate ones (with mixed Tate overlap), is itself mixed Tate.

\section{Mixed Tateness vis-\`a-vis Linearity and Chow-K\"unneth Generation}\label{section 3}

In this section we say a few words about the relation between mixed Tateness and other conditions which are currently prominent in the study of motives.

\subsection{...as Linearity sans Stratification}\label{subsection 3.1}

This work began as an attempt to replicate the results of Totaro's paper \cite{totaro2014chow}, especially those that pertain to toric varieties, in the context of cluster varieties. The main subject of that work is linear schemes.

\begin{defn}\label{linearity}
    A scheme over a field $k$ is \textit{linear} if it belongs to the smallest class of (separated, finite type) schemes over $k$ that 1) contains $\emptyset$ and $\mathbb{A}^n_k$ for all $n \geq 0$, 2) is closed under: if $X$ is linear and $i: X \hookrightarrow \mathbb{A}^n$ is a closed embedding then $\mathbb{A}^n \setminus X$ is linear, and 3) is closed under: if $X$ admits a finite decomposition $X = \bigcup_i S_i$ into locally closed linear schemes such that $\overline{S}_i \setminus S_i \subset \bigcup_{dim({S_j}) < dim(S_i)} S_j$, then $X$ is also linear.
\end{defn}

We will call a decomposition as in point 3) a \textit{stratification}, since this agrees with Totaro's notion \cite{totaro2014chow} (see Section 2 there). As examples, we note that $\mathbb{G}_m$ is linear over any field $k$, and that if $X$ is linear then $X \times \mathbb{A}^n$ is also linear.\\

The slogan of this subsection is that over a number field, the main theorems of \cite{totaro2014chow} remain true if one replaces ``linear scheme" with ``variety which is mixed Tate in the sense of Definition \ref{Mixed Tate}". The upshot is that the latter class is strictly broader: for example, as Totaro explains in Section 4 of \cite{totaro2016motive}, there exist Barlow surfaces which are mixed Tate but not linear.

\begin{prop}\label{linear implies MT}
   A linear scheme $X$ over any field $k$ is mixed Tate.
\end{prop}

\begin{proof}
    By iterated application of the excision sequence for compactly supported motives. $M^c(\mathbb{A}^n)$ is mixed Tate because it's a shift of a Tate twist, and the result then follows from the fact that $DMT(k)$ satisfies a two-out-of-three property for closed immersions, per Lemma \ref{2-of-3 mixed tate}.
\end{proof}

A related important input to our study is the following

\begin{prop}[Section 6 of \cite{totaro2016motive}]\label{strong CKgP}
    If a scheme $X$ is mixed Tate over a field $k$ (i.e. $M^c(X) \in DMT(k)$), then for any other scheme $Y$ the natural map \begin{equation}
        A_*(X) \otimes A_*(Y) \rightarrow A_*(X \times Y) \end{equation} \noindent is an isomorphism.
\end{prop}

We won't define or use Fulton and MacPherson's operational Chow ring, but the next fact is now sufficiently easy to state and prove that we list it anyway.

\begin{prop}[Mixed Tate Analogue of Theorem 2 of \cite{totaro2014chow}] For any mixed Tate variety $X$ which is proper over a field $k$, Fulton-MacPherson's (rationalized) operational Chow ring $A_{op}^*(X)$ satisfies \begin{equation}
    A_{op}^i(X) \cong \Hom(A_i(X), \mathbb{Q}),
\end{equation}
the $\mathbb{Q}$-linear dual of $A_i(X)$.
\end{prop}

\begin{proof}
    This follows from Propositions 1 and 2 of \cite{totaro2014chow} and the previous proposition.
\end{proof}

Now we turn in full force to theorems and properties regarding cycle class maps. From this point on, the assumption that $k$ is a number field is indispensable in this subsection.

\begin{thm}[Mixed Tate Analogue of Theorem 3 of \cite{totaro2014chow}]\label{cycle class iso}
    For any mixed Tate variety $X$ over a number field $k$ and for all $i$, \begin{equation}
        A_i(X) \cong \Hom(\mathbb{Q}(i)[2i], H_*^{BM}(X(\mathbb{C}))) \cong gr_{-2i}H_{2i}^{BM}(X(\mathbb{C}); \mathbb{Q})
    \end{equation} \noindent as $\mathbb{Q}$ vector spaces. Further, the cycle class map $cl_i := R^H_{k, \sigma}: A_i(X) \rightarrow gr_{-2i}H_{2i}^{BM}(X(\mathbb{C});\mathbb{Q})$ induced by any $R^H_{k, \sigma}$ witnesses this isomorphism.
\end{thm}

\begin{proof}
    By heredity of $MT(k)$, \begin{equation}
        A_i := \Hom(\mathbb{Q}(i)[2i], M^c(X)) \cong \Hom(\mathbb{Q}(i), H_{2i}^{M, BM}(X)) \oplus DMT^1(\mathbb{Q}(i), H_{2i+1}^{M, BM}(X)).
    \end{equation} In fact $DMT^1(\mathbb{Q}(i), H_{2i+1}^{M, BM}(X))$ vanishes. This is because $R^H_{k,\sigma}$ is compatible with weight filtrations and $H_{2i+1}^{M, BM}(X)$ has weight filtration $W$ with support in degrees only lying in the interval $[-2i - 1, 0]$, hence only in even degrees in $[-2i, 0]$ by the mixed Tateness assumption. This means that the associated graded of $W$ has no $\mathbb{Q}(j)$ factors for $j \geq i + 1$, which precludes nonzero Ext terms by Equation \ref{MT hereditary} and the discussion after it. An essentially identical argument shows that \begin{equation}
        \Hom(\mathbb{Q}(i)[2i], H_*^{BM}(X(\mathbb{C}))) \cong \Hom(\mathbb{Q}(i), W_{-2i}H_{2i}^{BM}(X(\mathbb{C}); \mathbb{Q})).
    \end{equation}
    Further, since $W_{-2i - 1} = 0$ for both objects, the previous claims upgrade to isomorphisms \begin{equation}
        A_i \cong \Hom(\mathbb{Q}(i), gr_{-2i}H_{2i}^{M, BM}(X))
    \end{equation} \noindent and \begin{equation}
        \Hom(\mathbb{Q}(i)[2i], H_*^{BM}(X(\mathbb{C}))) \cong \Hom(\mathbb{Q}(i), gr_{-2i}H_{2i}^{BM}(X(\mathbb{C}); \mathbb{Q})).
    \end{equation} The targets of both Hom sets decompose into direct sums of $\mathbb{Q}(i)$'s by definition, and we commented above that realization functors send $\mathbb{Q}(i)$'s to $\mathbb{Q}(i)$'s and act as the identity on their endomorphism spaces. Since $R^H_{k, \sigma}$ is exact and respects filtrations, i.e. sends $gr_{-2i}H_{2i}^{M, BM}(X)$ to $gr_{-2i}H_{2i}^{BM}(X(\mathbb{C});\mathbb{Q})$, this means we are done.
\end{proof}

The same argument using the compatibilities of $R^H_{k,\sigma}$ to ensure that $gr_{-2i}H_{2i+1}^{M,BM}$ and $gr_{-2i}H_{2i+1}^{BM}$ consist of the same finite number of copies of $\mathbb{Q}(i)$ in different categories (but in this case being unable to preclude the existence of $\text{Ext}^1$ terms in the source because of different supports in the relevant weight filtrations) implies the following

\begin{thm}[Mixed Tate Analogue of Totaro's ``Strong Property"]\label{strong property}
    For any mixed Tate variety $X$ over a number field $k$ and for all $i$, the cycle class map induces a surjection \begin{equation}
        cl := R^H_{k,\sigma}: \Hom(\mathbb{Q}(i)[2i+1], M^c(X))\twoheadrightarrow gr_{-2i}H_{2i+1}^{BM}(X(\mathbb{C});\mathbb{Q}),
    \end{equation}
    where the source of $cl$ is isomorphic to the higher Chow group $H_{2i+1}^{M, BM}(X; \mathbb{Q}(i))$.
\end{thm}

\subsection{...as Chow-K\"unneth Generation Beyond the Smooth and Proper Case}\label{subsection 3.2}

Recent work on the tautological ring of the moduli of curves has centered a property called the \textit{Chow-K{\"u}nneth generation property} (abbreviated as the \textit{CKgP}). Though we find that using homological Chow groups is more natural in the rest of the paper, throughout this section we use cohomological grading because it will allow us to write the theorems without grading shifts. We therefore write $A^*(X) := CH^*(X) \otimes \mathbb{Q}$.

\begin{defn}
    A variety $X$ over a field $k$ is said to \textit{have the CKgP} if for all algebraic stacks $Y$ of finite type admitting a stratification into global quotient stacks over $k$, the natural morphism \begin{equation}
        A^*(X) \otimes A^*(Y) \rightarrow A^*(X \times Y) \end{equation}\noindent is a surjection.
\end{defn}

\begin{lem}\label{schemes suffice}
    For a variety $X$ over a field $k$ to have the CKgP, it suffices to know that the natural morphism $A^*(X) \otimes A^*(Y) \rightarrow A^*(X \times Y)$ is surjective for all finite type separated schemes $Y$ over $k$.
\end{lem}

\begin{proof}
    Assume given an algebraic stack $Y$ satisfying the assumptions of the previous definition.
    We may reduce immediately to the case when $Y$ is just a global quotient stack by excision on a stratification.
    But the Chow groups of a global quotient stack are by definition just those of certain finite type separated $k$ schemes, per \cite{MR1614555}.
\end{proof}

For smooth and proper varieties, Canning and Larson show in \cite{MR5083265} that having the CKgP implies that the cycle class map is an isomorphism.

\begin{prop}[Lemma 3.11 of \cite{MR5083265}]
    If $X$ is smooth and proper over $\mathbb{C}$ and has the CKgP, then the cycle class map \begin{equation}
        cl: CH^*(X) \otimes \mathbb{Q} \rightarrow H^*(X;\mathbb{Q}) 
    \end{equation}
    \noindent is an isomorphism.
\end{prop}

Since mixed Tateness implies an isomorphism of the cycle class map onto the pure part of $H^*(X;\mathbb{Q})$ for $X$ over a number field, this suggests a relationship between the two conditions. We show now that, indeed, they are \textit{the same} for smooth and proper varieties.

\begin{thm}\label{mixed Tate vs. ckgp}
    Let $X$ be a smooth and proper variety over a field $k$. Then the following conditions are equivalent:

    \begin{enumerate}
        \item $M^c(X)$ is mixed Tate.
        \item $X$ has the CKgP.
        \item For every finitely generated field extension $F/k$, the morphism $A^*(X) \rightarrow A^*(X_F)$ induced by pullback along the flat cover $Spec(F) \rightarrow Spec(k)$ is a surjection.
    \end{enumerate}

    Further, mixed Tateness implies the remaining points even if $X$ is not smooth and proper.
\end{thm}

\begin{proof}
    $1$ implies $2$ by Proposition \ref{strong CKgP} and Lemma \ref{schemes suffice}. That $3$ implies $1$ for smooth and proper $X$ is Theorem 4.1 of \cite{totaro2016motive} -- we note that $3 \implies 1$ can genuinely fail for general varieties $X$.\\

    Therefore we are left to show that $2$ implies $3$. Fix a variety $Y$ over $k$ with function field $F$, and let $\{U_i\}$ be a cofiltered system of affine opens of $Y$ whose limit is Spec $F$. By the CKgP, for each $U_i$ we receive a surjective map $A^*(X) \otimes A^*(U_i) \rightarrow A^*(X \times U_i)$. Since $A^*$ commutes with cofiltered limits of schemes with affine transition maps, we see that in the limit we receive a map $A^*(X) \cong A^*(X) \otimes A^*(\Spec F) \rightarrow A^*(X_F)$. But a colimit of a commuting sequence of surjective morphisms is a surjection, so we're done.
\end{proof}

For non-smooth and proper $X$ this chain cannot be reversed.

\begin{exmp}
    Let $X = \mathbb{A}^2 \setminus E$, where $E$ is the affine part of a smooth genus $1$ curve. $X$ is certainly smooth but it is not proper. Since $E$ is not mixed Tate, $X$ is not either. However, $X$ satisfies the $CKgP$ by Lemma 3.3 of \cite{MR5083265}.
\end{exmp}

Indeed, while mixed Tateness seems reasonably useful as a condition even for non-smooth varieties, it does lack some of the flexibility in the smooth but possibly non-compact setting that the CKgP enjoys. While it plays well with stratifications, products and group quotients, the previous example shows that it does not persist under passage to opens. In particular, affine space minus a general hypersurface of sufficiently high degree will never be mixed Tate.\\

On the other hand, we are not aware of any standard way to exhibit a proper, non-smooth variety as having the CKgP without also showing that it is mixed Tate.

\section{Motivic Invariants of Cluster Varieties and their Compactifications}\label{section 4}

\subsection{Background on Cluster Varieties and Curious Hard Lefschetz}\label{subsection 4.1}

We refer the reader to Section 3 of \cite{galashin2026braid} for the definition of a cluster variety. We choose this reference in particular because we work in the generality of skew-symmetrizable cluster algebras. Therefore a cluster variety $\mathcal{A}(\Sigma)$ depends on a \textit{seed} $\Sigma = (\tilde{E}, \{d_i \in \mathbb{Z}_{>0}\}_{i \in [n + m]})$, where $\tilde{E}$ is an $(n + m) \times n$ \textit{extended exchange matrix} and \begin{equation}
    d_j \tilde{E}_{ij} = -d_i \tilde{E}_{ji} \:\text{for}\: i,j \in [n + m].
\end{equation} The definition of $\mathcal{A}(\Sigma)$ also depends on a choice of labeling of $[n + m]$ landing in $\mathbb{Q}(x_1, ..., x_{n+m})$ (a \textit{cluster}), which we will usually leave tacit. Nonetheless, any cluster variety is defined over $\mathbb{Z}$.\\

Thinking of $[n +m]$ as a vertex set $V$, a seed $\Sigma$ carries with it sets of mutable and frozen vertices $V \setminus F$ and $F$ of cardinality $n$ and $m$ (and tacit identifications $V \setminus F \cong [n]$ and $F \cong [n+m]\setminus[n]$). Therefore the rows of $\tilde{E}$ are indexed by $V$ and the columns of $\tilde{E}$ are indexed by $V \setminus F$. A seed $\Sigma$ is said to be \textit{skew-symmetric} if all of the $d_i$ are equal, in which case the information of $\Sigma$ is equivalent to that of an ice quiver $Q$ on vertex set $V$. Therefore we will often write $\mathcal{A}(Q)$ to denote the corresponding cluster variety of a skew-symmetric seed. We decline to define mutation of seeds.\\

To a seed $\Sigma$ we may associate a directed graph $\Gamma_\Sigma$ with vertex set $V \setminus F$ and an arrow $i \rightarrow j$ whenever $\tilde{E}_{ij} > 0$. Given $s \in V \setminus F$ we write $N_{in}(s)$ to denote the set of elements $v \in V \setminus F$ for which there is an edge from $v$ to $s$ in $\Gamma_\Sigma$. Further, given $I \subset V \setminus F$ we write $\Sigma^{\backslash I}$ to denote the seed induced from $\Sigma$  by freezing all the elements of $I$. A vertex $s \in \Gamma_\Sigma$ is called a \textit{sink} if $\tilde{E}_{vs} \geq 0$ for all $v \in V \setminus F$, i.e. if it is a sink of the graph $\Gamma_\Sigma$.\\

A seed $\Sigma$ or cluster variety $\mathcal{A}(\Sigma)$ is said to be \textit{full rank} (respectively \textit{really full rank}, often abbreviated to \textit{RFR}) if $\tilde{E}$ has full rank over $\mathbb{Q}$ (resp. the rows of $\tilde{E}$ span $\mathbb{Z}^n$). We will use freely that a seed being RFR is preserved under freezings of vertices. We say that $\mathcal{A}(\Sigma)$ is \textit{isolated} if $\Gamma_\Sigma$ is a discrete graph, and we say $\mathcal{A}(\Sigma)$ is \textit{acyclic} if $\Gamma_{\Sigma'}$ has no directed cycles for some mutation $\Sigma'$ of $\Sigma$. The seed of an isolated cluster variety is skew-symmetrizable by any choice of $\{d_i\}$, and so by fiat we declare that such seeds are skew-symmetric.\\

We say that $\mathcal{A}(\Sigma)$ is \textit{sink-recurrent} if $\Sigma$ is sink-recurrent, where the class of sink-recurrent seeds is defined recursively as follows.

\begin{itemize}
    \item Any isolated seed $\Sigma$ is sink-recurrent.
    \item Any seed that is mutation equivalent to a sink-recurrent seed is sink-recurrent.
    \item Suppose that there exists $s \in V \setminus F$ such that $s$ is a sink of $\Gamma_\Sigma$, and such that the seeds $\Sigma^{\backslash s}$ and $\Sigma^{\backslash N_{in}(s) \cup \{s\}}$ are sink-recurrent. Then $\Sigma$ is sink-recurrent.
\end{itemize}

If $s \in \Gamma_\Sigma$ is a sink satisfying the third condition above then we say that $\Sigma$ is \textit{sink-recurrent at $s$}. Being RFR or sink-recurrent is a mutation invariant of a seed $\Sigma$. Further, if $\Sigma$ is RFR and sink-recurrent at $s$ then $\mathcal{A}(\Sigma^{\backslash s})$ and $\mathcal{A}(\Sigma^{\backslash N_{in}(s)})$ cover $\mathcal{A}(\Sigma)$ with overlap $\mathcal{A}(\Sigma^{\backslash N_{in}(s) \cup \{s\}})$. In this case the sink-recurrence of $\mathcal{A}(\Sigma^{\backslash N_{in}(s) \cup \{s\}})$ implies that $\mathcal{A}(\Sigma^{\backslash N_{in}(s)})$ is also sink-recurrent.\\

RFR sink-recurrent cluster varieties are locally acyclic, and therefore they are smooth over a characteristic $0$ field by the main theorem of \cite{muller2013locally}. It turns out that their motivic invariants are both remarkably simple and exhibit interesting combinatorial structures. Indeed, in the RFR, Louise, and skew-symmetric case, Lam and Speyer \cite{lam2022cohomology} showed the following package of theorems.

\begin{thm}[Main Theorems of \cite{lam2022cohomology}]\label{louise CHL}
    Any RFR, Louise, and skew-symmetric cluster variety $\mathcal{A}(Q)$ over $\mathbb{C}$ has $H^{k, (p, q)}(\mathcal{A}(Q)) = 0$ for all $p \neq q$ (i.e., $\mathcal{A}(Q)$ has mixed Tate cohomology). Further, the weight filtration on each $H^k(\mathcal{A}(Q))$ is split over $\mathbb{Q}$.
\end{thm}

They show this by exhibiting the existence, for any $2e$-dimensional $\mathcal{A}(Q)$, of a $2$-form $\gamma$ such that $[\gamma] \in H^{2, (2,2)}(\mathcal{A}(Q))$ and \begin{equation}\label{CHL maps}
    [\gamma]^{e - p}:H^{p + s, (p, p)}(\mathcal{A}(Q)) \rightarrow H^{2e - p + s, (2e - p, 2e - p)}(\mathcal{A}(Q))
\end{equation}
\noindent is an isomorphism. By taking the product of an odd-dimensional $\mathcal{A}(Q)$ with $\mathbb{C}^*$ this is enough to show that a similar isomorphism (though, notably, one not actually induced by cupping with any $[\gamma] \in H^{2, (2, 2)}(\mathcal{A}(Q))$) exists between $H^{p + s, (p, p)}$ and $H^{d - p + s, (d - p, d - p})$ for any $d$-dimensional RFR, Louise, skew-symmetric $\mathcal{A}(Q)$.\\

A variety $X$ equipped with a $\gamma$ as in Equation \ref{CHL maps} is said to satisfy the \textit{curious hard Lefschetz} property (with respect to $\gamma$). Further, one which simply has the corresponding panoply of isomorphisms between mixed Hodge groups is said to have \textit{curious Poincar\'e symmetry}. Of particular computational interest to us in the sequel is the fact that if $X$ has curious Poincar\'e symmetry then $$h^{2i, (i, i)}(X) = h^{d, (d-i, d-i)}(X)$$ \noindent for all $i$.\\

Theorem \ref{louise CHL} extends to the more general setting of cluster varieties which are only RFR and sink-recurrent.

\begin{prop}\label{sink-recurrent CHL}
     Any complex RFR sink-recurrent cluster variety $\mathcal{A}(\Sigma)$ has curious Poincar\'e symmetry, mixed Tate cohomology, and $\mathbb{Q}$-split weight filtrations.
\end{prop}

\begin{proof}
Per Theorem 10.1 of \cite{galashin2022braid} and Footnote 6 of \cite{galashin2026braid}, one sees that any even-dimensional RFR sink-recurrent cluster variety $\mathcal{A}(\Sigma)$ satisfies the curious hard Lefschetz property (with respect to some $2$-form $\gamma$ such that $[\gamma] \in H^{2, (2, 2)}(\mathcal{A}(\Sigma))$). Therefore any RFR sink-recurrent cluster variety has curious Poincar\'e symmetry. Theorem 3.3 of \cite{lam2022cohomology} then implies that $H^*(\mathcal{A}(\Sigma); \mathbb{Q})$ has split weight filtrations.\\

Mixed Tateness of the cohomology follows from Corollary \ref{sink-recurrent MT} applied to $\mathcal{A}(\Sigma)$ considered as a variety over $\mathbb{Q}$ followed by Hodge realization.
\end{proof}

Our main example of a class of RFR sink-recurrent cluster varieties is that of \textit{(double) braid varieties}, a class containing open Richardson varieties $R_{u,w}^\circ$ in flag varieties of arbitrary Lie type. Once again we refer the reader to the first two sections of \cite{galashin2026braid} for definitions and general theory. Double braid varieties admit cluster structures which are RFR and sink-recurrent by Theorem 4.10 and Corollary 6.7 of \cite{galashin2026braid}, and these cluster structures are defined over $\mathbb{Z}$ by Section 2.4 of \cite{casals2025comparing}.\\

We will only actually do computations involving (double) braid varieties in type $A$, and then we will only work over $\mathbb{C}$ (though again, every variety and cluster structure that follows is defined over $\mathbb{Z}$). By the main theorems of \cite{casals2025comparing}, this allows us to discuss them using the formalism developed in \cite{MR4868947}. Therefore, we write $\beta \in$ $Br_n^+$ to denote a positive braid or braid word on $n$ strands with Demazure product $\delta(\beta) = w_0 \in S_n$, and write $\Delta := \sigma_1(\sigma_2 \sigma_1)\cdots(\sigma_{n-1}\dots \sigma_1)$ to denote the half-twist. One must fix an \textit{expression} for $\beta$ in order to define a braid variety, but we often do so tacitly because the resulting isomorphism class of variety does not depend on this choice (though each expression for $\beta$ induces a possibly different cluster structure on $X(\beta)$).\\

Let $i \in [n-1]$. Following Definition 2.1 of \cite{MR4868947}, we define the \textit{braid matrix} $B_i(z)$ associated to the twist $\sigma_i$ by
\begin{equation} B_i(z) = I_{i-1} \oplus B(z) \oplus I_{n - i - 1},
\end{equation}
\noindent where $B(z) = \begin{bmatrix}
    0      & 1  \\
    1       & z
\end{bmatrix}.$
Therefore $B_i(z)$ is the matrix which is the identity away from a $2 \times 2$ subblock with upper left corner at $(i,i)$ which equals $B(z)$. Given $\beta = \sigma_{i_1} \cdots \sigma_{i_r} \in Br^+_n$ and $(z_1, ..., z_r) \in \mathbb{C}^r$, we write 
\begin{equation}
    B_\beta(z) := B_{i_1}(z_1) \cdots B_{i_r}(z_r).
\end{equation}
\begin{defn}
    The \textit{braid variety} $X(\beta)$ associated to a positive braid word $\beta = \sigma_{i_1} \cdots \sigma_{i_r}$ is the subvariety of $\mathbb{C}^r$ on which $B_\beta(z)w_0$ is upper triangular, i.e.
    \begin{equation}
        X(\beta) := \{(z_1, ..., z_r) \in \mathbb{C}^r\:|\:B_\beta(z)w_0 \in \mathcal{B}\}.
    \end{equation}
\end{defn}

$X(\beta)$ is a $(\ell(\beta) - \ell(\Delta))$-dimensional affine variety cut out by the vanishing of generalized minors of $B_\beta(z)$. Accordingly, we will often refer to the quantity $\ell(\beta) - \ell(\Delta)$ as $d(\beta)$, or even just $d$ when clear from context, in the sequel. By Corollary 5.8 and Theorem 5.9 of \cite{galashin2022braid}, any $X(\beta)$ admits the structure of a (skew-symmetric) RFR and sink-recurrent cluster variety.

\subsection{Cluster Structures on Contracted Seeds and Linearity of Cluster Varieties}\label{subsection 4.2}

Our preliminary goal is to show that every RFR sink-recurrent cluster variety and many compactifications thereof are linear schemes over any field $k$. Since linear schemes are mixed Tate by Proposition \ref{linear implies MT}, in the characteristic $0$ setting this provides a motivic lift (and weakens the assumptions) of the first main theorem of \cite{lam2022cohomology}, which says that complex really full rank, Louise and skew-symmetric cluster varieties have mixed Tate mixed Hodge structures.\\

We will exhibit RFR sink-recurrent cluster varieties as linear by constructing a ``freezing-contraction" stratification of them into things which are products of RFR sink-recurrent cluster varieties with affine spaces. Namely, if $\Sigma$ is sink-recurrent at $s$ and if $x_s$ is the corresponding cluster variable on $\mathcal{A}(\Sigma)$, we show that $\mathcal{A}(\Sigma^{\backslash s}) \cong D(x_s)$ and $V(x_s)$ are both products of RFR sink-recurrent cluster varieties with affine spaces.\\

The stratification we will discuss was essentially already considered by \cite{galashin2026braid}, in particular in their Theorem 3.13 where it is used to \textit{produce} cluster structures on varieties\footnote{In \cite{galashin2026braid} they refer to this stratification by the name ``deletion-contraction", in analogy with the kinds of recursions one does in matroid theory. Since we consider seeds where we do all of freezing, literally deleting, and contracting a vertex, we employ different terminology. In particular, what they call deletions we call freezings, and we refer to a different class of seeds as deletions, see Definition \ref{delete contract}.}. Our contribution is a sort of converse: that if one begins with an RFR sink-recurrent cluster variety, then in fact the ``contracted", i.e. $V(x_s)$, portion of the stratification is $\mathbb{A}^1$ times an RFR sink-recurrent cluster variety too. The relevant seed and a sibling are constructed as follows.

\begin{con}
    Let $\Sigma$ be a really full rank sink-recurrent seed, and let $s \in \Gamma_\Sigma$ be a sink. We define  seeds $\Sigma^{-s}$ and $\Sigma^{/s}$ via the following steps.

    \begin{itemize}
        \item First, start with $\Sigma^{\backslash N_{in}(s)}$, which has $(n + m) \times (n - |N_{in}(s)|)$ exchange matrix $\tilde{E}$.
        \item Since the row corresponding to the sink $s$ is null in $\tilde{E}$, we may delete this row from $\tilde{E}$ to define an $(n + m - 1) \times (n - |N_{in}(s)|)$ matrix $\tilde{E}^{red}$.
        \item Next, we delete the column corresponding to $s$ from $\tilde{E}^{red}$ (i.e. we restrict the source of $\tilde{E}^{red}$ away from the span of $\mathbb{Z}e_s$). This produces an $(n + m - 1) \times (n - |N_{in}(s)| - 1)$ matrix $\tilde{E}^{-s}$.
        \item Finally, we quotient the target of $\tilde{E}^{-s}$ by $\mathbb{Z}\chi_s := \tilde{E}^{red}(\mathbb{Z}e_s)$. Since $\Sigma$ is RFR, $\chi_s$ is a primitive vector; so by composing we receive an $(n + m - 2) \times (n - |N_{in}(s)| - 1)$ matrix $\tilde{E}^{/s}$.
    \end{itemize}
\end{con}

\begin{defn}\label{delete contract}
     The \textit{deletion} of $s$ from $\Sigma$ is the seed defined by \begin{equation}
         \Sigma^{-s} := (\tilde{E}^{-s}, \{d_i\}_{i \neq s}),
     \end{equation} where $\{d_i\}$ is the set of positive integers in the definition of $\Sigma$. It is the seed arising from deleting the vertex $s$ from $\Sigma^{\backslash N_{in}(s)}$. Similarly, the \textit{contraction} $\Sigma^{/s}$ of $s$ in $\Sigma$ is the seed defined by \begin{equation}
         \Sigma^{/s} = (\tilde{E}^{/s}, \{d_i\}_{i \neq s}).
     \end{equation}
     
     We note that all of the quotients on the target in the above construction only affect the frozen parts of the corresponding seeds, and so both $\Sigma^{-s}$ and $\Sigma^{/s}$ have the same mutable part as $\Sigma^{\backslash N_{in}(s) \cup \{s\}}$.
\end{defn}

\begin{rem}
    Given any RFR seed $\Sigma$ and $s \in \Gamma_\Sigma$, one may \textit{always} form deleted and contracted seeds $\Sigma^{-s}$ and $\Sigma^{/s}$. In this case, one defines $\Sigma^{-s}$ and $\Sigma^{/s}$'s exchange matrices by freezing the entire neighborhood $N(s)$ of $s$ in $\Gamma_\Sigma$, deleting the (null) row of the exchange matrix $\tilde{E}$ of $\Sigma^{\backslash N(s)}$ corresponding to $s$, and then restricting the source (to get $\Sigma^{-s}$) and also quotienting the target (to get $\Sigma^{/s}$) of the resulting matrix by $\mathbb{Z}e_s$ and $\mathbb{Z}\chi_s$. This generality is unnecessary in the present work but is considered in the following proposition.
\end{rem}

\begin{prop}\label{still rfr}
    Assume given a really full rank seed $\Sigma$ and $s \in \Gamma_\Sigma$. Then $\Sigma^{-s}$ and $\Sigma^{/s}$ (in the sense of the previous remark) are RFR. If $\Sigma$ is also sink-recurrent at $s$, then $\Sigma^{-s}$ and $\Sigma^{/s}$ are RFR and sink-recurrent.
\end{prop}

\begin{proof}
     Since $\tilde{E}$ is an isomorphism onto a saturated sublattice of its target and deleting a null row from a matrix does not change its span, we see that $\tilde{E}^{red}$ is also an isomorphism onto a saturated sublattice of $\mathbb{Z}^{n + m - 1}$. This means that $\tilde{E}^{red}$ is a split injection, and so we may fix a decomposition of the target of $\tilde{E}^{red}$ into a direct sum \begin{equation}
        \mathbb{Z}^{n + m - 1} \cong C \oplus \mathbb{Z}\chi_s \oplus A
    \end{equation}
    where $C \cong coker(\tilde{E}^{red})$ and $\mathbb{Z}\chi_s \oplus A \cong im(\tilde{E}^{red})$, so that $A$ is the span of the columns of $\tilde{E}^{red}$ corresponding to $e_i$ with $i \neq s$. Therefore the quotient of $\mathbb{Z}^{n + m - 1}$ by $\mathbb{Z}\chi_s$ restricts to an isomorphism on $A$, and so if one composes $\tilde{E}^{red}$ with a restriction of its source away from $\mathbb{Z}e_s$ and a quotient of its target by $\mathbb{Z}\chi_s$ then the resulting matrix induces an isomorphism onto a saturated sublattice of $\mathbb{Z}^{n + m - 2}$. But by the definition of $A$ the matrices of the restriction and of the restriction plus quotient, respectively, are the matrices $\tilde{E}^{-s}$ and $\tilde{E}^{/s}$, and so $\Sigma^{-s}$ and $\Sigma^{/s}$ are really full rank.\\

    To see that $\Sigma^{-s}$ and $\Sigma^{/s}$ are sink-recurrent if $\Sigma$ is sink-recurrent at $s$, we note that sink-recurrence of a seed $\Sigma$ only depends on its mutable part. This follows from a standard induction on the number of mutable vertices in $\Sigma$. Therefore the sink-recurrence of $\Sigma^{-s}$ and $\Sigma^{/s}$ follows from that of $\Sigma^{\backslash N_{in}(s) \cup \{s\}}$, which is sink-recurrent by assumption.
\end{proof}

We assume from now on that $\Sigma$ is really full rank and sink-recurrent. The point we are approaching is that if $\Sigma$ is sink-recurrent at $s$, then $\mathcal{A}(\Sigma^{/s})$ is very closely related to $V(x_s) \subset \mathcal{A}(\Sigma)$. To demonstrate this we recall a bit more of the geometric structure of deleted and contracted seeds. Assume once again that $\Sigma$ is sink-recurrent at a sink $s$. First, since $\Sigma^{\backslash N_{in}(s)}$ is $\Sigma^{-s}$ with the vertex $s$ still present, the clusters of the former are just those of the latter with a choice of $x_s$ or its mutation $x_s'$. That is: \begin{equation}\label{conic eq}
    H^0(\mathcal{A}(\Sigma^{\backslash N_{in}(s)})) = H^0(\mathcal{A}(\Sigma^{-s}))[x_s, x'_s]/(x_sx'_s = \chi'_s + 1),
\end{equation}
where $\chi'_s$ denotes the non-trivial portion of the exchange binomial corresponding to the sink $s$. In particular we deduce that $\mathcal{A}(\Sigma^{\backslash N_{in}(s)})$ admits an affine binomial fibration over $\mathcal{A}(\Sigma^{-s})$.\\

Next, we claim that $\mathcal{A}(\Sigma^{/s})$ can be identified with the subvariety of $\mathcal{A}(\Sigma^{-s})$ on which $\chi'_s + 1$ vanishes. Indeed, since $\Sigma^{-s}$ is RFR and $\chi_s$ is a primitive element of its frozen lattice, up to a quasi-equivalence we may consider $\chi_s$ as a frozen vertex of $\Sigma^{-s}$. Since $\Sigma^{/s}$ is precisely the seed one gets from quotienting the frozen lattice of $\Sigma^{-s}$ by $\chi_s$, the claim follows from Proposition 3.7 in \cite{muller2013locally}, Proposition \ref{still rfr}, and the local acyclicity of sink-recurrent seeds.\\

Lastly, we note that $\mathcal{A}(\Sigma^{\backslash N_{in}(s) \cup \{s\}}) \cong \mathcal{A}(\Sigma^{-s}) \times \mathbb{G}_m$ because $\Sigma^{\backslash N_{in}(s) \cup \{s\}}$ is just $\Sigma^{-s}$ with an extra isolated frozen vertex. In total, the geometry of the situation described in the last few paragraphs is depicted in Figure \ref{conic fig}. With these facts in mind we may prove the following

\begin{prop}\label{vanishing locus is linear}
    Let $\Sigma$ be an RFR seed which is sink-recurrent at $s \in \Gamma_\Sigma$, and let $x_s$ be the corresponding cluster variable on $\mathcal{A}(\Sigma)$. Then $V(x_s) \subset \mathcal{A}(\Sigma)$ is isomorphic to $\mathbb{A}^1 \times \mathcal{A}(\Sigma^{/s})$.
\end{prop}

\begin{proof}
    To start we note that $V(x_s) \subset \mathcal{A}(\Sigma^{\backslash N_{in}(s)}) \subset \mathcal{A}(\Sigma)$, where the latter inclusion is a cluster localization because $\Sigma^{\backslash N_{in}(s)}$ is locally acyclic. The first inclusion follows from the fact that if $s$ is a sink and $x_s = 0$ then the exchange relation implies that all of the cluster variables corresponding to vertices incident to $s$ are invertible.\\

    Therefore we may consider $V(x_s x'_s)$ as the reduced preimage of $V(\chi'_s + 1)$ under the morphism from $\mathcal{A}(\Sigma^{\backslash N_{in}(s)})$ to $\mathcal{A}(\Sigma^{-s})$ described above. Setting $x'_s = 0$ allows $x_s$ to vary freely (and vice-versa), and so we see that the two irreducible components $V(x_s)$ and $V(x'_s)$ of this preimage are relative $\mathbb{A}^1$'s. As we have noted, $V(\chi'_s + 1)$ is precisely $\mathcal{A}(\Sigma^{/s})$; and affine space torsors over affine schemes are all trivial. This completes the proof.
    \end{proof}


    \begin{figure}
        \centering
        \includegraphics[width=.7\linewidth]{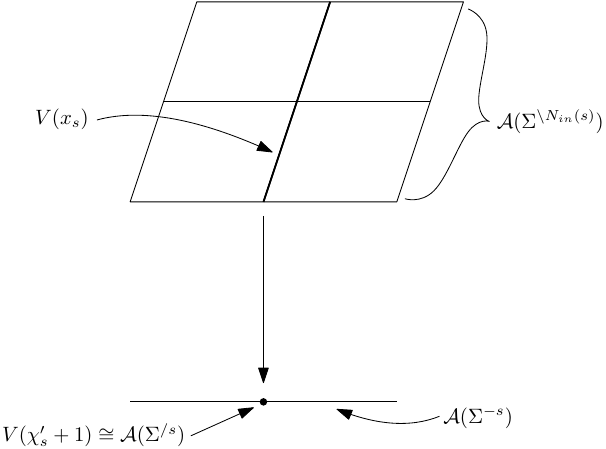}
        \caption{An illustration of the geometric argument in Equation \ref{conic eq} and Proposition \ref{vanishing locus is linear}. If an RFR seed $\Sigma$ is sink-recurrent at $s$, then $\mathcal{A}(\Sigma^{\backslash N_{in}(s)})$ fibers over $\mathcal{A}(\Sigma^{-s})$. The vanishing locus of $\chi'_s + 1 = 0$ is $\mathcal{A}(\Sigma^{/s})$, and the reduced preimage of this locus is exactly $V(x_s) \cup V(x'_s)$.}
        \label{conic fig}
    \end{figure}

\begin{thm}\label{RFR LINEAR THM}
    A really full rank sink-recurrent cluster variety $\mathcal{A}(\Sigma)$ is a linear scheme over any field $k$.
\end{thm}

\begin{proof}
    By induction on the number of vertices in $\Gamma_\Sigma$. If there are none then $\mathcal{A}(\Sigma)$ is a (split) torus, and so is certainly linear.\\

    In general, let us assume that we have mutated $\Sigma$ so that it is sink-recurrent at some vertex $s \in \Gamma_\Sigma$. Then by the previous proposition we learn that \begin{equation}\label{sink decomp}
        \mathcal{A}(\Sigma) = D(x_s) \sqcup V(x_s) \cong \mathcal{A}(\Sigma^{\backslash s}) \sqcup (\mathbb{A}^1 \times \mathcal{A}(\Sigma^{/s})).
    \end{equation}

    Both of the two cluster varieties appearing on the right hand side have fewer mutable vertices than $\Sigma$, and so they are both linear by the induction hypothesis. Further, $\overline{D(x_s)} \setminus D(x_s) = V(x_s)$ is clearly of lower dimension than $D(x_S)$. Therefore the right hand side is a stratification (in the sense of Definition \ref{linearity}) of $\mathcal{A}(\Sigma)$ into linear schemes, and so $\mathcal{A}(\Sigma)$ is linear by definition.
\end{proof}

\begin{cor}\label{sink-recurrent MT}
    The compactly supported motive of a really full rank sink-recurrent cluster variety $\mathcal{A}(\Sigma)$ over any field $k$ is mixed Tate.
\end{cor}

    

In particular this means that any (double) braid variety $X(\beta)$, and therefore any open Richardson variety $R_{u,w}^\circ$ in a flag variety of arbitrary Lie type, is linear over any field $k$, per Section 2.4 of \cite{casals2025comparing}.\\

As a proof of concept we deduce a set of corollaries. For example, since any flag Richardson variety $R_{u,w}$, i.e. the closure of a Richardson cell $R_{u,w}^\circ$ inside of the flag variety, is stratified into a disjoint union of open Richardson varieties, this implies the following.

\begin{cor}\label{first linear}
    A flag Richardson variety $R_{u,w}$ is linear over any field $k$.
\end{cor}

Similarly, by composing with the projection $\pi: G/B \rightarrow G/P$ and setting
\begin{equation}
    \Pi_{u,w} := \pi(R_{u,w}),
\end{equation} we may apply the results of \cite{MR3176610} (which in particular show that $\Pi_{u,w}$ is stratified into a disjoint union of isomorphic images of $R_{u',w'}^\circ$'s) to deduce the following

\begin{cor}
    A projected Richardson variety $\Pi_{u,w}$ is a linear scheme over any field $k$.
\end{cor}

\begin{rem}
    By Theorem \ref{mixed Tate vs. ckgp} the previous two corollaries imply that all flag Richardson and projected Richardson varieties (in arbitrary type and over any field) have the CKgP.
\end{rem}

Restricting to braid varieties of braid $\beta$ in type $A$, this shows that all brick varieties (in the sense of \cite{MR3512647}) are linear, because by \cite{MR5110495} said varieties are stratified into a disjoint union of braid varieties of subwords of $\beta$. It is unclear to us that this follows immediately from the existence of a Bia\l ynicki-Birula (BB) decomposition of the brick variety, since such decompositions need not satisfy Totaro's boundary condition on stratification: this fails for BB decompositions of certain type $B_n, n \geq 4$ Bott-Samelson varieties by Proposition 2.5 of \cite{dudas2008note}, whose statement about Deodhar cells in a double Schubert cell lifts to the corresponding one on the Bott-Samelson.

\begin{cor}\label{last linear}
    A brick variety brick($\beta$) is linear over any field $k$.
\end{cor}

Of course, the BB decomposition suffices in place of the corollary wherever the latter is used below to compare Chow and cohomology groups.

\begin{rem}
    If $Q$ is an RFR acyclic quiver, then iterating Equation \ref{sink decomp} at the sinks of $Q$ (without mutating!) recovers the anticlique stratification of $\mathcal{A}(Q)$ studied in \cite{LSII}.
\end{rem}

\subsection{Splitness of RFR Sink-Recurrent Cluster Varieties}\label{subsection 4.3}

In this subsection we show that RFR sink-recurrent cluster varieties over number fields have motives which are simpler still: they are all split. This provides a motivic lift of the second part of the main theorem of \cite{lam2022cohomology}. The key lemma for showing this is the following statement about motives.

\begin{prop}
       Assume given $M \in MT(\mathbb{Q})$ such that $R^H(M)$ has split weight filtration, i.e. $R^H(M)$ decomposes as a direct sum of Tate twists. Then $M$ is split in the sense of Definition \ref{Mixed Tate}.
\end{prop}

\begin{proof}
    By induction on the length of the weight filtration $W$ of $M$. The base case is when $W$ has length $1$, in which case $M \cong \mathbb{Q}(i)^{\oplus r}$ for some $i$ and $r$ by definition.\\
    
    Now suppose that $W$ has length $\ell$. Then one receives a short exact sequence \begin{equation}
        0 \rightarrow W^{\ell-1}(M) \rightarrow M \rightarrow gr_\ell^W(M) \rightarrow 0
\end{equation}\noindent in $MT(\mathbb{Q})$. By induction, to see that $M$ is split it suffices to show that the class $\gamma$ of the above extension in $\text{Ext}^1(gr_\ell^W(M), W^{\ell-1}(M))$ is $0$.\\

    Since $R^H$ is exact we see that the image of the above SES is again an SES in $MHS_\mathbb{Q}$, and functoriality dictates that $R^H(\gamma) \in \text{Ext}^1(R^H(gr_\ell^W(M)), R^H(W^{\ell-1}(M)))$ is the extension class corresponding to this new SES. But this extension splits by assumption. Therefore $R^H(\gamma) = 0$, and the injectivity of $R^H$ on $\text{Ext}$'s implies that $\gamma = 0$, i.e. our original SES splits.
\end{proof}

\begin{thm}
    Let $X$ be a mixed Tate variety defined over $\mathbb{Q}$, and assume that for every $i$ the weight filtration on $H_i^{BM}(X(\mathbb{C}); \mathbb{Q})$ is split over $\mathbb{Q}$. Then $X_k$ is split over any number field $k$.
\end{thm}

\begin{proof}
If $k = \mathbb{Q}$ then the claim follows from the previous proposition and the assumption that each $H_i^{BM}(X(\mathbb{C}))$ is split. For general $k$ the theorem then follows from the facts that $M^c(X)$ is isomorphic to the direct sum of its cohomologies (by Equation \ref{MT hereditary}), that $M^c(X_k) \in DMT(k)$ is the pullback of $M^c(X)$ under the base change functor from $DMT(\mathbb{Q})$ to $DMT(k)$, and that Tate twists pull back to Tate twists.
\end{proof}

\begin{thm}\label{sink-recurrent split}
    A really full rank sink-recurrent cluster variety $\mathcal{A}(\Sigma)$ is split over any number field $k$.
\end{thm}

\begin{proof}
    By Proposition \ref{sink-recurrent CHL}, Corollary \ref{sink-recurrent MT}, and Poincar\'e duality on $\mathcal{A}(\Sigma)(\mathbb{C})$, the conditions of the previous theorem apply to $H_i^{BM}(\mathcal{A}(\Sigma)(\mathbb{C});\mathbb{Q})$.
\end{proof}

For any prime $\ell$ there is an \'etale realization functor $F_\ell$ from $DMT(k)$ to $\ell$-adic representations of $Gal(\bar{k}/k)$ which sends $\mathbb{Q}(i)$ to $\mathbb{Q}_\ell(i)$, the one-dimensional representation where $Gal(\bar{k}/k)$ acts via the $i$-th cyclotomic character, and which sends $M_c(X)$ to the Galois representation $H_*^{BM}(X_{\bar{k}}, \mathbb{Q}_\ell)$. By utilizing these realizations we learn the following

\begin{cor}
    Fix a number field $k$ and an RFR sink-recurrent cluster variety $\mathcal{A}(\Sigma)$. Then, for any prime $\ell$ and $i \geq 0$, the Galois representation $Gal(\bar{k}/k) \curvearrowright H_c^{i, \acute{e}t}(\mathcal{A}(\Sigma)_{\overline{k}}; \mathbb{Q}_\ell)$ splits into a direct sum of tensor powers of the cyclotomic character. The number of terms and the characters which appear in $H_c^{i, \acute{e}t}(\mathcal{A}(\Sigma)_{\bar{k}}; \mathbb{Q}_\ell)$ are independent of $\ell$.
\end{cor}

Similarly, using splitness, Equations \ref{MT hereditary} and \ref{MT} imply that the higher Chow groups of a really full rank sink-recurrent cluster variety $\mathcal{A}(\Sigma)$ over a number field $k$ are remarkably simple. In particular they decompose rationally into sums of mixed Hodge pieces of the cohomology of $\mathcal{A}(\Sigma)$ and the higher Chow groups of $k$. We leave the interested reader to work out the details, and restrict comments in this (and related) directions to the following remark.

\begin{rem}
   Subsection \ref{subsection 4.2} suffices to show that really full rank sink-recurrent cluster varieties are \textit{stably cellular} over any field $k$, meaning that the compactly supported motivic spectra associated to them in Voevodsky's category $SH(k)$ of motivic spectra over $k$ are constructed from spheres by iterated pushouts.\\
    
    That linearity implies stable cellularity was already commented by Totaro in \cite{totaro2016motive}. We stress this connection, however, because we expect that the Chow vanishing theorems we establish in the next section for RFR sink-recurrent cluster varieties can serve as an input to computational tools which are available for cellular schemes and which can potentially compute complicated motivic invariants of RFR sink-recurrent cluster varieties. An example of the type of results we have in mind is given in the following
    
    \begin{cor}
        $K_0(\mathcal{A}(\Sigma)) = \mathbb{Z}$ for any RFR sink-recurrent cluster variety over a characteristic $0$ field $k$.
    \end{cor}
    
    \begin{proof}
        This follows from Theorem \ref{chow/cohom theorem}, because the associated graded pieces of the coniveau filtration on $G_0$ are subquotients of Chow groups $CH_i(\mathcal{A}(\Sigma))$. But any $\mathcal{A}(\Sigma)$ is smooth in characteristic $0$ by \cite{muller2013locally}, and so we have that $G_0(\mathcal{A}(\Sigma)) \cong K_0(\mathcal{A}(\Sigma))$.
    \end{proof}
    
    In turn, we hope that combining computations of motivic invariants with tools like Theorems \ref{cycle class iso} and \ref{strong property} can be used to deduce further representation-theoretic results, as in works like \cite{barkley2026combinatorial}. For a survey on stable cellularity and its consequences we refer the reader to \cite{Hlavinka-REU}.
\end{rem}

\section{Homological and Chow Vanishing for Cluster Varieties}\label{section 5}


By the top cohomology of a smooth affine complex variety $X$ of dimension $d$ we will mean the \textit{$d$-th} cohomology of $X$ rather than the $2d$-th. This is justified by the Artin vanishing theorem. The Chow statements in this section will hold over any field $k$, while the cohomological ones will be made over $\mathbb{C}$. Since we will only work with varieties which we have shown to be split over $\mathbb{Q}$, however, we will often write (e.g.) $\mathcal{A}(\Sigma)$ instead of $\mathcal{A}(\Sigma)(\mathbb{C})$.\\

Before proving the general theorem, as a warmup we indicate an application of Theorem \ref{cycle class iso} which upgrades results which are already in the literature to a rational Chow and cohomological vanishing theorem. Theorem \ref{chow/cohom theorem} will dispense with the assumptions of acyclicity, skew-symmetry, and rationality made in the following discussion.\\

By the ``main corollary" (Corollary 1.4) of \cite{LSII} it is known that any acyclic, RFR and skew-symmetric cluster variety $\mathcal{A}(Q)$ has top cohomology satisfying $h^d(\mathcal{A}(Q)) = 1$ and \begin{equation}\label{LS cor}
    H^d(\mathcal{A}(Q)) = H^{d, (d, d)}(\mathcal{A}(Q)).
\end{equation} Further, curious Poincar\'e symmetry implies that \begin{equation}
    H^d(\mathcal{A}(Q)) \cong \bigoplus_{\frac{d}{2} \leq i \leq d} H^{d, (i, i)}(\mathcal{A}(Q)) \cong \bigoplus_{2i \leq d} H^{2i, (i, i)}(\mathcal{A}(Q)),
\end{equation} \noindent where said property applies to show that $H^{2i, (i, i)}(\mathcal{A}(Q)) \cong H^{d, (d - i, d - i)}(\mathcal{A}(Q))$. Therefore Theorem \ref{cycle class iso} allows us to immediately deduce the following

\begin{prop}
    The rational Chow groups $A_i$ of a $d$-dimensional acyclic, RFR and skew-symmetric cluster variety $\mathcal{A}(Q)$ over $\mathbb{Q}$ vanish away from top degree, i.e. \begin{equation}
        A_i(\mathcal{A}(Q)) = \begin{cases}  0 & i < d \\
    \mathbb{Q} & i = d.
   \end{cases}
    \end{equation}
\end{prop}

\begin{proof}
    Equation \ref{LS cor}, i.e. Corollary 1.4 of \cite{LSII}, shows that $H^{d, (i, i)}(\mathcal{A}(Q)(\mathbb{C})) = 0$ for all $i < d$. Therefore curious Poincar\'e symmetry shows that $H^{2i, (i, i)}(\mathcal{A}(Q)(\mathbb{C}))$ vanishes for all $i > 0$. Since $\mathcal{A}(Q)(\mathbb{C})$ is smooth, $H^{2i, (i, i)}(\mathcal{A}(Q)(\mathbb{C}))$ is isomorphic to $H_{2d -2i, (i - d, i - d)}^{BM}(\mathcal{A}(Q)(\mathbb{C}))$; and by Theorem \ref{cycle class iso} and Corollary \ref{sink-recurrent MT} the latter group is isomorphic to $A_{d-i}(\mathcal{A}(Q))$. The $i = d$ case follows from irreducibility of $\mathcal{A}(Q)$.
\end{proof}

The main theorem of this section is a deduction of the same Chow vanishing \textit{integrally} and for \textit{all} RFR sink-recurrent cluster varieties.

\begin{thm}\label{chow/cohom theorem}
    For any $d$-dimensional RFR sink-recurrent cluster variety $\mathcal{A}(\Sigma)$ over a field $k$,
     \begin{equation}
        CH_i(\mathcal{A}(\Sigma)) = \begin{cases}  0 & i < d \\
    \mathbb{Z} & i = d.
   \end{cases} \end{equation} Therefore $h^d(\mathcal{A}(\Sigma)) = 1$ and $
    H^d(\mathcal{A}(\Sigma)) = H^{d, (d, d)}(\mathcal{A}(\Sigma))$ when $k = \mathbb{C}$.
\end{thm}

\begin{proof}
Since RFR sink-recurrent cluster varieties are mixed Tate, by Theorem \ref{cycle class iso} and curious Poincar\'e symmetry the second statement follows from the first. To prove the first we induct on the number of mutable vertices in $\Sigma$. When there are none then $\mathcal{A}(\Sigma)$ is a split torus and the theorem is certainly true.\\

In general we may mutate to assume that $\Sigma$ is sink-recurrent at a sink $s$. Then by Proposition \ref{vanishing locus is linear}, we may decompose $\mathcal{A}(\Sigma)$ as \begin{equation}
    \mathcal{A}(\Sigma) = D(x_s) \sqcup V(x_s) \cong \mathcal{A}(\Sigma^{\backslash s}) \sqcup (\mathbb{A}^1 \times \mathcal{A}(\Sigma^{/s})).
\end{equation}
Therefore the excision sequence in Chow homology takes the form of a right exact sequence \begin{equation}
    CH_{i}(\mathbb{A}^1 \times \mathcal{A}(\Sigma^{/s})) \rightarrow CH_i(\mathcal{A}(\Sigma)) \rightarrow CH_i(\mathcal{A}(\Sigma^{\backslash s})) \rightarrow 0,
\end{equation}
or equivalently
\begin{equation}
    CH_{i-1}(\mathcal{A}(\Sigma^{/s})) \rightarrow CH_i(\mathcal{A}(\Sigma)) \rightarrow CH_i(\mathcal{A}(\Sigma^{\backslash s})) \rightarrow 0.
\end{equation}
The outer terms are Chow groups of cluster varieties of seeds with fewer mutable vertices than $\Sigma$ by construction, with $\mathcal{A}(\Sigma^{/s})$ being $(d - 2)$-dimensional and $\mathcal{A}(\Sigma^{\backslash s})$ being $d$-dimensional. Since both have fewer mutable vertices than $\Sigma$ we see that the outer groups vanish for all $i < {d-1}$ by induction, and so $CH_i(\mathcal{A}(\Sigma)) = 0$ in the same range. Further, the image of the fundamental class of $\mathbb{A}^1 \times \mathcal{A}(\Sigma^{/s})$ is $[div(x_s)]$ in $CH_{d-1}(\mathcal{A}(\Sigma))$. This implies that
$CH_{d-1}(\mathcal{A}(\Sigma)) = 0$. The final case that $CH_d(\mathcal{A}(\Sigma)) = \mathbb{Z}$ follows from irreducibility of the variety in question.
\end{proof}

In particular we learn the following

\begin{cor}
    The integral Chow groups of a (double) braid variety $X(\beta)$ or open projected Richardson variety $\Pi_{u,w}^\circ$ over a field vanish away from top (homological) degree.
\end{cor}

\begin{rem}\label{top cohom remark}
    Theorem \ref{chow/cohom theorem} implies that the pullback map $H^d(\mathcal{A}(\Sigma)) \rightarrow H^d((\mathbb{C}^*)^d)$ is an isomorphism for any RFR sink-recurrent $\Sigma$ and any cluster torus inside the corresponding complex cluster variety $\mathcal{A}(\Sigma)$, generalizing Corollary 1.4 of \cite{LSII}. We refer to their work for potential applications.
\end{rem}

\begin{figure}
    \centering \includegraphics[width=\linewidth]{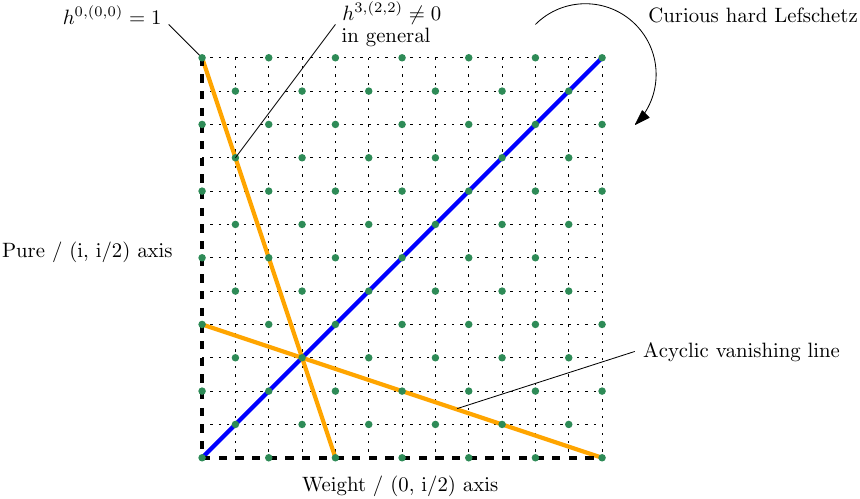}
    \caption{The green dots represent potentially non-zero mixed Hodge cohomology groups of an RFR sink-recurrent complex cluster variety $\mathcal{A}(\Sigma)$. The group represented by the dot $i$ steps down and $j$ steps to the right of the top-left point is $H^{i, (\frac{i + j}{2}, \frac{i + j}{2})}(\mathcal{A}(\Sigma)).$
    Artin vanishing shows that there are no non-zero such groups below the ``Weight" axis (so named because moving left and right in this diagram only changes weights), while weight theory shows there are none to the left of the ``Pure" axis (so named because the groups along the vertical dotted line are the pure cohomology groups).}\label{cohomology chart}
\end{figure}

Figure \ref{cohomology chart} contains a diagram illustrating the homological information deduced over the last few years about general RFR sink-recurrent complex cluster varieties. There are no cohomologies ``off" of this diagram because any RFR sink-recurrent $\mathcal{A}(\Sigma)$ has mixed Tate cohomology, and the groups are symmetric about the blue line by curious Poincar\'e symmetry, both by Proposition \ref{sink-recurrent CHL}. If $Q$ is acyclic then in fact all groups below either of the solid orange lines vanish, by the main theorems of \cite{LSII}. The orange lines are solid, however, because in general there are non-zero groups along them: for example, Proposition 9.4 of \cite{LSII} shows that $H^{3, (2, 2)}(\mathcal{A}(Q)) = H^1(\underline{Q}^{mut}; \mathbb{Q})$ if $Q$ is acyclic, where $\underline{Q}^{mut}$ denotes the underlying topological space of the \textit{undirected} induced subgraph on the mutable portion of $Q$.\\

Previous methods did not show whether the axes (the thicker black dashed lines) are ``solid" or not, i.e. whether one can have non-zero pure cohomology groups for RFR sink-recurrent cluster varieties. Theorem \ref{chow/cohom theorem} shows that all such groups are indeed trivial (hence the \textit{dotting} of the axes) save the extremal ones, which are $1$-dimensional.

\section{Consequences of Vanishing and Splitness}\label{section 6}

\subsection{...in Knot Homology}\label{subsection 6.1}

The tools developed in the preceding sections allow us to deduce the vanishing of a suite of knot homology groups of positive braid closures. All varieties in the first subsection are assumed over the complex numbers unless otherwise stated.\\

Let us establish some notation. From here out we write $d(\beta) = \ell(\beta) - \ell(\Delta)$, or even just $d$ when the input braid is clear, to denote the dimension of the braid variety associated to a positive braid $\beta$.  Given $\beta \in Br_n^+$ a positive braid on $n$ strands, we write $HHH(\beta)$ to denote the \textit{(reduced) Khovanov-Rozansky homology} of $\beta$. Our conventions for its definition are the same as in \cite{li2025last}, which are somewhat more common than those in \cite{MR4809331}. In particular, Li uses the convention that the Rouquier complex of a single letter $\sigma_s$ is $F_{\sigma_s}^\bullet := (B_s \rightarrow R(1))$ (where $R$ is a polynomial ring over $\mathbb{Q}$) while Galashin and Lam set $F_{\sigma_s}^{\bullet, GL} := (B_s \rightarrow R)$.\\

In any event, $HHH(\beta)$ arises as the cohomology of a certain Hochschild complex associated to $\beta$ and is triply graded, by Hochschild degree $a$, cohomological degree $t$, and quantum degree $q$, and we write $HHH^{a, t, q}(\beta)$ for the group corresponding to a given grade.\\

We recall that the equivariant (compactly supported or Borel-Moore) (co)homology of a variety $X$ equipped with the action of an algebraic torus admits a mixed Hodge structure, and that Section 2.2 of \cite{MR5110495} shows that any braid variety of the form $X(\beta), \beta \in Br_n^+$ admits an action of $T := (\mathbb{C}^*)^{n-1}$. Our work so far and knot homology groups are then connected by the following

\begin{thm}[Theorem 4.22 of \cite{MR5110495}]\label{HHH equivariant cohom}
Given a positive braid $\gamma$ with $\delta(\gamma) = w_0$, one has
    \begin{equation}
        gr^WH_{c, T}^*(X(\gamma)) \cong HHH^{a = 0}(\gamma \Delta^{-1}).
    \end{equation}
\end{thm}

Note that our notation differs slightly from that of \cite{MR5110495} in that we do not use the reversal of $\gamma$ on the left hand side. The two varieties are $T$-equivariantly isomorphic up to a twist, however, so nothing is lost in this replacement.\footnote{This was clarified to us by Eugene Gorsky and Soyeon Kim.}\\

The above isomorphism is, indeed, one of (bi)graded abelian groups, with respect to gradings which we now spend some time discussing. We recall that with Galashin and Lam's conventions for Rouquier complexes, given a type $A$ open flag Richardson variety $R_{v, w}^\circ$ thought of as the braid variety of some positive braid $\gamma$, one receives an isomorphism \begin{equation}
    H_{c, T}^{d + 2p + k, (p, p)}(R_{v, w}^\circ) \cong H^{k, (p)}(HH^0(F_{\gamma \Delta^{-1}}^{\bullet, GL})).
\end{equation}
In particular this means that $H^{d, (p, p)}_{c, T}(R_{v,w}^\circ) \cong H^{-2p, (p)}(HH^0(F_{\gamma \Delta^{-1}}^{\bullet, GL}))$. Since Li gives the generators of $R$ degree $2$ and twists $F_{\sigma_s}^\bullet$ by an extra degree (both cohomological and quantum), one sees that \begin{equation}
    H^{k,(p)}(HH^0(F_{\gamma \Delta^{-1}}^{\bullet, GL})) \cong H^{k + d, (2p - d)}(HH^0(F_{\gamma \Delta^{-1}}^\bullet)),
\end{equation}
\noindent and so $H^{d + 2p +k, (p, p)}_{c, T}(R_{v, w}^\circ) \cong H^{k + d, (2p - d)}(HH^0(F_{\gamma \Delta^{-1}}^\bullet))$. Relabeling again we recover the appropriate graded version of Theorem \ref{HHH equivariant cohom}, and indeed this is the correct grading in general: \begin{equation}
   H^{t + q + d, (\frac{q + d}{2}, \frac{q + d}{2})}_{c, T}(X(\gamma)) \cong HHH^{0, t, q}(\gamma \Delta^{-1}).
\end{equation}

We will usually apply the previous equation in the case that $\gamma = \beta \Delta$, in which case $d(\beta \Delta) = \ell(\beta)$. Let us also note suggestively that sans the subscript ``$T$" the $q = -t$ graded groups on the left hand side of the above are (isomorphic to) the groups which we computed in Theorem \ref{chow/cohom theorem} -- i.e., under Poincar\'e duality they identify with the groups along the dotted Weight axis of Figure \ref{cohomology chart}.\\

We give an example to illustrate this.

\begin{exmp}\label{ex:trefoil}
Let $\beta = \sigma_1^4 \in \mathrm{Br}^+_2$. In this case the $T$ action on $X(\beta)$ is free and the
quotient is a two-dimensional variety $Y$, so that $H^{k, (p, q)}_c(Y) \cong H^{k + 1, (p, q)}_{c, T}(X)$ (as $dim(T) = n - 1 = 1$). By Example 4.20 of \cite{MR4809331}, the only nonzero mixed Hodge groups of Y are
\begin{equation}
    H_c^{4,(2,2)}(Y) = \mathbb{Q}\;\;\text{and}\;\; H_c^{2,(0,0)}(Y) = \mathbb{Q},
\end{equation} and so $H^{5, (2, 2)}_{c, T}(X) \cong H^{3, (0, 0)}_{c, T}(X) \cong \mathbb{Q}$. On the other hand, Section 2.5 of \cite{li2025last} unwinds to tell us that $HHH^{a = 0}(\sigma_1^3)$ is supported in degrees $(t,q) = (1, 1)$ and $(3,-3)$, which agrees with the above.
\end{exmp}

Our main use for this is the numerical consequence that the line $t + q = 0$ in \cite{li2025last}'s notation is the same as the line $k + 2p = 0$ in \cite{MR4809331}'s, and in particular that Galashin-Lam's $H^{d, (p,p)}_{c, T}(X(\beta)) \cong HHH^{0, d(\beta) - 2p, -d + 2p}(\beta \Delta^{-1})$.\\

With the relevant grading comparisons now established we may prove the following

\begin{thm}\label{HHH vanishing}
Let $\beta \in Br_n^+$ be a positive braid on $n$ strands. Then $HHH^{0, t, -t}(\beta)$ vanishes for all $t \neq \ell(\beta).$ Further, $HHH^{0, \ell(\beta), -\ell(\beta)}(\beta) = \mathbb{Q}$.
\end{thm}

\begin{proof}
    The preceding discussion tells us that it suffices to show instead that $H_{c, T}^d(X(\beta \Delta)) \cong H_{c, T}^{d, (0, 0)}(X(\beta \Delta))$ and that this group is one-dimensional. To do this we consider the structure morphism $(\mathbb{C}^N \setminus \{0\})^{n - 1} \rightarrow (\mathbb{P}^{N-1})^{n - 1}$ induced by the total space of the tautological bundle for large $N$, which approximates the universal bundle $ET \rightarrow BT$ in the sense that for any CW complex $X$ we have  \begin{equation}
        H^*_{c, T}(X) \cong H^*_c(X \times _T (\mathbb{C}^N \setminus \{0\})^{n - 1})
    \end{equation} for $N \gg *$. Further, $(\mathbb{C}^N \setminus \{0\})^{n - 1} \rightarrow (\mathbb{P}^{N-1})^{n - 1}$ is a Zariski locally trivial fibration, and so any time we have an affine variety $X$ with a $T$ action, $X \times_T (\mathbb{C}^N \setminus \{0\})^{n - 1}$ also exists as a quasi-projective variety by Proposition 23 of \cite{MR1614555}. In particular this implies that the Leray spectral sequence of the fibration from $X(\beta \Delta) \times_T (\mathbb{C}^N \setminus \{0\})^{n - 1}$ to $(\mathbb{P}^{N-1})^{n-1}$ with fiber $X(\beta \Delta)$ is in fact a spectral sequence of mixed Hodge structures by the main theorem of \cite{MR2178703}.\\

   The compactly supported Leray spectral sequence of this fibration has $E_2$ page \begin{equation}
        E_2^{p, q} = H^p_c((\mathbb{P}^{N-1})^{n-1}, \underline{H^q_c(X(\beta \Delta))})
    \end{equation} and converges to $H^{p + q}_c(X(\beta \Delta) \times_T (\mathbb{C}^N \setminus \{0\})^{n - 1})$. Since $(\mathbb{P}^{N-1})^{n-1}$ is simply connected, each of the local systems $\underline{H^q_c(X(\beta \Delta))}$ is trivial. Therefore, since our computations are with field coefficients, we see that $E_2^{p, q}$ may in fact be identified with $H^p((\mathbb{P}^{N-1})^{n-1}) \otimes H_c^q(X(\beta \Delta))$ (as a mixed Hodge structure, and using that $(\mathbb{P}^{N-1})^{n-1}$ is compact).\\

    From Artin vanishing we therefore learn that $E_2^{p,q} = 0$ for all $q < d$ and that $E_2^{0, d} \cong H^{d}_c(X(\beta \Delta))$, where the latter is an isomorphism of mixed Hodge structures. This means that no differentials enter or leave $E_2^{0, d} \cong H^{d}_c(X(\beta \Delta))$. This being the only group on the $d$-th diagonal of the $E_2$ page, we immediately learn that $H_c^d(X(\beta \Delta) \times_T (\mathbb
    C^N \setminus \{0\})^{n-1}) \cong H^{d}_c(X(\beta \Delta))$ as mixed Hodge structures. Now the claim follows from Theorem \ref{chow/cohom theorem} upon taking sufficiently large $N$ and using Poincar\'e duality.
    \end{proof}

This computation can be extended to an extra degree if \begin{equation}\label{treelike vanishing}
    H^{d-1}(X(\beta \Delta)) \cong \oplus_{0 \leq k \leq i} H^{d-1, (d-1 - k, d-1 - k)}(X(\beta \Delta)).
\end{equation} Under this assumption, the same argument as above will show that $H^{d+1, (k, k)}_{c, T}(X(\beta \Delta))$ vanishes for all $k \geq i + 2$: the only new point is that there is a differential $E^{0, d+1}_2 \rightarrow E_2^{2, d}$, where the target is pure of weight $2$ by inspection, and so whose kernel contains all summands of $E^{0, d+1}_2$ besides $H^{d+1, (1, 1)}_c(X(\beta \Delta))$. Therefore we learn from Equation \ref{treelike vanishing} and Poincar\'{e} duality that $HHH^{0, t, 1-t}(\beta)$ vanishes for all $t < \ell(\beta) - 2i - 1$.\\

If $Q$ is RFR acyclic, then $H^{d-1}(\mathcal{A}(Q)) \cong H^{d-1, (d-1, d-1)}(\mathcal{A}(Q)) \oplus H^{d-1, (d-2, d-2)}(\mathcal{A}(Q))$ by Theorem 1.3 of \cite{LSII} and curious Poincar\'{e} symmetry. If $Q$ is moreover a \textit{forest quiver}, i.e. an RFR quiver such that the underlying undirected graph of $Q^{mut}$ has no cycles, then also $H^{d-1, (d-2, d-2)}(\mathcal{A}(Q)) = 0$ by Proposition 9.4 of \cite{LSII} and curious Poincar\'e again. This is to say that an analogue of Equation \ref{treelike vanishing} holds with $i = 1$ and $i = 0$ for the cluster varieties of RFR acyclic and RFR forest quivers, respectively. More generally, we expect that Equation \ref{treelike vanishing} holds with $i = 0$ for any RFR leaf-recurrent cluster variety (in the sense of \cite{galashin2024plabic}). Per Proposition 2.10 there, this class includes all type $A$ positroid cells $\Pi^\circ_f$.

\begin{conj}\label{leaf-recurrent conj}
    If $\mathcal{A}(Q)$ is an RFR leaf-recurrent cluster variety, then Equation \ref{treelike vanishing} holds for $i  = 0$.
\end{conj}

All in all we deduce the following facts.

\begin{cor}\label{forest vanishing}
     Fix $\beta \in Br_n^+$ a positive braid on $n$ strands, and assume that $X(\beta \Delta)$ admits a cluster structure with respect to an RFR acyclic quiver (respectively an RFR forest quiver). Then $HHH^{0, t, 1-t}(\beta)$ vanishes for all $t < \ell(\beta) - 3$ (resp. for all $t < \ell(\beta) - 1$).
\end{cor}

\begin{conj}\label{positroid braid conj}
   For $\beta \in Br^+_n$, assume that $X(\beta \Delta)$ admits a cluster structure with respect to an RFR leaf-recurrent quiver. Then $HHH^{0, t, 1-t}(\beta)$ vanishes for all $t < \ell(\beta) - 1$.
\end{conj}

The forest quiver portion of the corollary above is a preliminary Khovanov-Rozansky complement to results in \cite{schwartz2026homfly}. On the other hand, the conjectures apply to any positroid braid $\beta_f$ in the sense of \cite{MR5110495}, since in this case $X(\beta_f \Delta)$ is isomorphic to the positroid cell $\Pi^\circ_f$ up to torus factors.\\

We now discuss another application of Theorem \ref{chow/cohom theorem}. Let $c(\beta)$ denote the number of connected components of $\Lambda(\beta)$, the $(-1)$-framed closure of $\beta$ into a link. Then, Section 2.2 of \cite{MR5110495} constructs a subtorus $T_{c(\beta)} \subset T$ of dimension $n - c(\beta)$ which acts freely on $X(\beta \Delta)$. The quotient $X(\beta \Delta)/T_{c(\beta)}$ turns out to be an interesting variety in its own right.
\begin{defn}[Theorem 2.39 of \cite{MR4855860}]
    The \textit{augmentation variety} of the $(-1)$-framed closure of a braid $\beta$ is \begin{equation}
        Aug(\beta, \mathfrak{t}_c) := X(\beta \Delta)/T_{c(\beta)}.
    \end{equation}
\end{defn}

Classically, the augmentation variety of a Legendrian knot $\Lambda$ equipped with a collection of marked points $\mathfrak{t}$ is defined to be the space of DGA morphisms from the Chekanov-Eliashberg DGA of $(\Lambda, \mathfrak{t})$ to $\mathbb{C}$. That we can instead define augmentation varieties as above is the content of the cited theorem.\\

In fact the proof of Theorem 2.30(ii) of \cite{MR5110495} implies that the quotient $X(\beta \Delta) \rightarrow Aug(\beta, \mathfrak{t}_c)$ is a trivial torus fibration. Therefore we learn the following

\begin{thm}\label{augmentation vanishing}
    The compactly supported motive of an augmentation variety $Aug(\beta, \mathfrak{t}_c)$ over any number field where the torus fibration $X(\beta \Delta) \rightarrow Aug(\beta, \mathfrak{t}_c)$ trivializes is split. Further, if $e$ is the dimension of $Aug(\beta, \mathfrak{t}_c)$, then \begin{equation}
        H^e(Aug(\beta, \mathfrak{t}_c)) \cong H^{e, (e, e)}(Aug(\beta, \mathfrak{t}_c)),
    \end{equation} and both groups are $1$-dimensional. Moreover, $$CH_i(Aug(\beta, \mathfrak{t}_c)) = 0$$ for all $i < e$ (with the same assumptions).
\end{thm}

\begin{proof}
    Some such number field exists because $T_{c(\beta)}$ is the kernel of characters on $T$. Extending to this base and writing \begin{equation}\label{augmentation splitting}
        X(\beta \Delta) = Aug(\beta, \mathfrak{t}_c) \times T_{c(\beta)}
    \end{equation} \noindent as described above, the symmetric monoidality of the functor sending a variety to its compactly supported motive then shows that \begin{equation}
        M^c(X(\beta \Delta)) \cong M^c(Aug(\beta, \mathfrak{t}_c)) \otimes M^c(T_{c(\beta)}).
    \end{equation} But the former variety is split by Theorem \ref{sink-recurrent split}, and $M^c(T_{c(\beta)})$ admits a shift of the monoidal unit as a summand. By the idempotent completeness of $MT(k)$, this means that $M^c(Aug(\beta, \mathfrak{t}_c))$ must also be split.\\

    The homological consequences follow from the K\"unneth formula applied to the $\mathbb{C}$-points of Equation \ref{augmentation splitting}. In particular, since both $X(\beta \Delta)$ and $T_{c(\beta)}$ have $1$-dimensional highest cohomology supported only in top weight the same must be true of $Aug(\beta, \mathfrak{t}_c)$. The Chow case follows from noting that since the fibration $X(\beta \Delta) \rightarrow Aug(\beta, \mathfrak{t}_c)$ is split, $CH_*(Aug(\beta, \mathfrak{t}_c))$ is a retract of $CH_*(X(\beta \Delta))$. 
\end{proof}

\subsection{...for the Cohomology of Compactifications}\label{subsection 6.2}

Because of the topological complexity of flag Richardson cells and positroid cells (which, as we recall, we consider in arbitrary Lie type), it is not \textit{a priori} clear that their compactifications have Chow groups or pure cohomologies generated by relevant boundary strata. Due to our results, this nevertheless turns out to be the case.

\begin{prop}
    Let $X$ be a flag Richardson variety, projected Richardson variety, or brick variety (respectively) over a field, and let $\{Z_i\}$ denote the set of boundary divisors of $X$ with respect to its Richardson, positroid, or braid stratification (resp.). Then for every $j < d = dim(X)$, the proper pushforward $$\bigoplus_i CH_j(Z_i) \rightarrow CH_j(X)$$ is surjective, and so the same holds for the lowest weight rational Borel-Moore homology groups of $X(\mathbb{C})$ by Theorem \ref{cycle class iso}.
\end{prop}

\begin{proof}
    By Corollaries \ref{first linear} through \ref{last linear} and Theorem \ref{cycle class iso} it suffices to show the Chow case. But this follows from excision, which tells us that for all $j < d$ we have a right exact sequence \begin{equation}
        \bigoplus_i CH_j(Z_i) \rightarrow CH_j(X) \rightarrow CH_j(X \setminus \bigcup_i Z_i) \rightarrow 0,
    \end{equation} and the latter term is $0$ because $X \setminus \cup_i Z_i$ is a double braid variety, i.e. an RFR sink-recurrent cluster variety.
\end{proof}

One can iterate this procedure on each of the $Z_i$, and then on each boundary divisor in a $Z_i$, and so on. Since each closed stratum of one of the above stratifications is again a (projected) Richardson or brick variety, at each step we deduce that everything but the fundamental class of a closed stratum in $CH_*$ is in the image of proper pushforwards from smaller strata. Through this we learn the following

\begin{thm}\label{generation}
    Let $X$ be a flag Richardson variety, projected Richardson variety, or brick variety (resp.) over a field, and let $\{S_i\}$ denote the set of closed strata of its Richardson, positroid, or braid stratification (resp.). Then $CH_*(X)$ is generated by the classes $[S_i]$, and the same is true of the lowest weight rational Borel-Moore homology groups of $X(\mathbb{C})$.
\end{thm}

\begin{rem}
   As examples of projected Richardson varieties, the previous two theorems apply to all closed Richardson varieties in a not-necessarily complete flag variety. Theorem \ref{generation} is \textit{not} the same as generation of the Chow groups of a Richardson variety by its Richardson strata in general.
\end{rem}

Of course, there are often far more boundary divisors in such an $X$ than one actually needs to generate a given homology group. For example, if $X$ is set to be the type A positroid variety $\Pi_{e,w_0} = Gr(k,n)$ then $H^2(X; \mathbb{Q}) \cong \mathbb{Q}$, but $Gr(k,n)$ has $n$ positroid divisors. This is more subtle than the Schubert case, where the images of sub-Schubert varieties form a basis of $A_*$. This leads us to the following question.

\begin{ques}
    What (combinatorial, geometric, etc.) objects do the Chow groups of a closed Richardson variety $R_{u,w}$, positroid variety $\Pi_{u,w}$, or brick variety brick($\beta$) parametrize? In other words, for any $i$ and appropriate $X$, what sets of Bruhat intervals, positroid pairs, or braid words index basis elements of $CH_i(X)$?
\end{ques}

For example, just as the answer to this question for a Schubert variety $X_w$ is ``elements of Bruhat order under $w$", the answer in the positroid case should involve an interesting collection of $P$-Bruhat intervals which are comparable to $[u, w]$. Similarly, \cite{bossinger26braid} supplies a natural \textit{spanning set} for the Picard groups of brick varieties in terms of certain subwords of $\beta$, and we expect there to be combinatorially interesting choices of  spanning or parametrizing sets for all Chow degrees.\\

In the positroid case, one way we see to approach this question is by degenerating $\Pi_{u,w}$ to a semi-toric variety and using a judicious choice of faces in the corresponding polyhedral complex to select a basis worth of smaller positroid strata in $\Pi_{u,w}$, a technique which is made possible in type A through the construction of such semi-toric degenerations in \cite{chang2026fence}.

\bibliographystyle{amsplain}
\bibliography{refs}

@article{lam2022cohomology,
  title={Cohomology of cluster varieties, {I}: Locally acyclic case},
  author={Lam, Thomas and Speyer, David E},
  journal={Algebra \& Number Theory},
  volume={16},
  number={1},
  pages={179--230},
  year={2022},
  publisher={Mathematical Sciences Publishers}
}

@article{LSII,
  title={Cohomology of cluster varieties {II}: Acyclic case},
  author={Lam, Thomas and Speyer, David E},
  journal={Journal of the London Mathematical Society},
  volume={108},
  number={6},
  pages={2377--2414},
  year={2023},
  publisher={Wiley Online Library}
}

@article {MR4868947,
    AUTHOR = {Casals, Roger and Gorsky, Eugene and Gorsky, Mikhail and Le,
              Ian and Shen, Linhui and Simental, Jos\'e},
     TITLE = {Cluster structures on braid varieties},
   JOURNAL = {J. Amer. Math. Soc.},
  FJOURNAL = {Journal of the American Mathematical Society},
    VOLUME = {38},
      YEAR = {2025},
    NUMBER = {2},
     PAGES = {369--479},
      ISSN = {0894-0347,1088-6834},
   MRCLASS = {13F60 (14M15 20F36)},
  MRNUMBER = {4868947},
MRREVIEWER = {Ashish\ K.\ Srivastava},
       DOI = {10.1090/jams/1048},
       URL = {https://doi.org/10.1090/jams/1048},
}

@misc{galashin2022braid,
  title={Braid variety cluster structures, {I}: 3{D} plabic graphs},
  author={Galashin, Pavel and Lam, Thomas and Sherman-Bennett, Melissa and Speyer, David E},
  eprint = {2210.04778},
  archivePrefix = {arXiv},
  year={2022},
  note = {\textup{arXiv:}\texttt{2210.04778}}
}

@misc{bossinger26braid,
  title={Torus actions on compactified braid varieties and polytopality of subword complexes},
  author={Bossinger, Lara and Gorsky, Mikhail and Simental, Jos{\'e}},
  eprint = {2609.12414},
  archivePrefix = {arXiv},
  year={2026},
  note = {\textup{arXiv:}\texttt{2609.12414}}
}

@article{GL,
  title={Mirror symmetry for truncated cluster varieties},
  author={Gammage, Benjamin and Le, Ian and others},
  journal={SIGMA. Symmetry, Integrability and Geometry: Methods and Applications},
  volume={18},
  pages={055},
  year={2022},
  publisher={SIGMA. Symmetry, Integrability and Geometry: Methods and Applications}
}

@misc{dupont2024introduction,
  title={An introduction to mixed {T}ate motives},
  author={Dupont, Cl{\'e}ment},
  eprint = {2404.03770},
  archivePrefix = {arXiv},
  year={2024},
  note = {\textup{arXiv:}\texttt{2404.03770}}
}

@inproceedings{deligne2005groupes,
  title={Groupes fondamentaux motiviques de {T}ate mixte},
  author={Deligne, Pierre and Goncharov, Alexander B},
  booktitle={Annales scientifiques de l’Ecole normale sup{\'e}rieure},
  volume={38},
  number={1},
  pages={1--56},
  year={2005},
  organization={Elsevier}
}

@inproceedings{totaro2014chow,
  title={Chow groups, {C}how cohomology, and linear varieties},
  author={Totaro, Burt},
  booktitle={Forum of Mathematics, Sigma},
  volume={2},
  pages={e17},
  year={2014},
  organization={Cambridge University Press}
}

@article{totaro2016motive,
  title={The motive of a classifying space},
  author={Totaro, Burt},
  journal={Geometry \& Topology},
  volume={20},
  number={4},
  pages={2079--2133},
  year={2016},
  publisher={Mathematical Sciences Publishers}
}

@article {MR5083265,
    AUTHOR = {Canning, Samir and Larson, Hannah},
     TITLE = {On the {C}how and cohomology rings of moduli spaces of stable
              curves},
   JOURNAL = {J. Eur. Math. Soc. (JEMS)},
  FJOURNAL = {Journal of the European Mathematical Society (JEMS)},
    VOLUME = {28},
      YEAR = {2026},
    NUMBER = {9},
     PAGES = {3869--3918},
      ISSN = {1435-9855,1435-9863},
   MRCLASS = {14C17 (14C15)},
  MRNUMBER = {5083265},
       DOI = {10.4171/jems/1543},
       URL = {https://doi.org/10.4171/jems/1543},
}

@article{galashin2026braid,
  title={Braid variety cluster structures, {II}: general type},
  author={Galashin, Pavel and Lam, Thomas and Sherman-Bennett, Melissa},
  journal={Inventiones Mathematicae},
  volume={243},
  number={3},
  pages={1079--1127},
  year={2026},
  publisher={Springer}
}

@article{muller2013locally,
  title={Locally acyclic cluster algebras},
  author={Muller, Greg},
  journal={Advances in Mathematics},
  volume={233},
  number={1},
  pages={207--247},
  year={2013},
  publisher={Elsevier}
}

@article {MR5110495,
    AUTHOR = {Casals, Roger and Gorsky, Eugene and Gorsky, Mikhail and
              Simental, Jos\'e},
     TITLE = {Positroid links and braid varieties},
   JOURNAL = {J. Reine Angew. Math.},
  FJOURNAL = {Journal f\"ur die Reine und Angewandte Mathematik. [Crelle's
              Journal]},
    VOLUME = {837},
      YEAR = {2026},
     PAGES = {1--54},
      ISSN = {0075-4102,1435-5345},
   MRCLASS = {14M15 (05 13 20 57K43)},
  MRNUMBER = {5110495},
       DOI = {10.1515/crelle-2026-0007},
       URL = {https://doi.org/10.1515/crelle-2026-0007},
}

@misc{Hlavinka-REU,
  author = {Hlavinka, Josephine},
  title  = {Motivic homotopy theory and cellular schemes},
  year   = {2022},
  note   = {REU paper, University of Chicago Mathematics REU},
  howpublished = {\url{https://math.uchicago.edu/~may/REU2022/REUPapers/Hlavinka.pdf}}
}

@article{cao2024valuation,
  title={The valuation pairing on an upper cluster algebra},
  author={Cao, Peigen and Keller, Bernhard and Qin, Fan},
  journal={J. Reine Angew. Math.},
  fjournal={Journal f{\"u}r die reine und angewandte Mathematik (Crelles Journal)},
  number={806},
  volume={2024},
  year={2024},
  pages={71--114},
  publisher={De Gruyter}
}

@article {MR4809331,
    AUTHOR = {Galashin, Pavel and Lam, Thomas},
     TITLE = {Positroids, knots, and {$q,t$}-{C}atalan numbers},
   JOURNAL = {Duke Math. J.},
  FJOURNAL = {Duke Mathematical Journal},
    VOLUME = {173},
      YEAR = {2024},
    NUMBER = {11},
     PAGES = {2117--2195},
      ISSN = {0012-7094,1547-7398},
   MRCLASS = {14M15 (05A15 05B35 13F60 14F08 57K18)},
  MRNUMBER = {4809331},
MRREVIEWER = {K.\ R.\ Goodearl},
       DOI = {10.1215/00127094-2023-0049},
       URL = {https://doi.org/10.1215/00127094-2023-0049},
}

@misc{li2025last,
  title={The Last Three T-degrees in Triply-Graded Link Homology},
  author={Li, Cailan},
  eprint = {2505.14182},
  archivePrefix = {arXiV},
  year={2025},
  note = {\textup{arXiv:}\texttt{2505.14182}}
}

@misc{chang2026fence,
  title={Fence Complexes and Toric Degenerations of Positroid Varieties},
  author={Chang, Cameron and Enugandla, Pranav and Hlavinka, Josephine},
  eprint = {2606.12815},
  archivePrefix = {arXiv},
  year={2026},
  note = {\textup{arXiv:}\texttt{2606.12815}}
}

@book{levine1998mixed,
  title={Mixed motives},
  author={Levine, Marc},
  number={57},
  year={1998},
  publisher={American Mathematical Soc.}
}

@misc{barkley2026combinatorial,
  title={Combinatorial invariance for the coefficient of $ q $ in {K}azhdan-{L}usztig polynomials},
  author={Barkley, Grant T and Gaetz, Christian and Lam, Thomas},
  eprint = {2601.07793},
  archivePrefix = {arXiv},
  year={2026},
  note = {\textup{arXiv:}\texttt{2601.07793}}
}

@misc{kelly2026some,
  title={Some explicit counter-examples to {W}eibel's conjecture},
  author={Kelly, Shane},
  eprint = {2608.16066},
  year={2026},
  note = {\textup{arXiv:}\texttt{2608.16066}}
}

@article {MR4855860,
    AUTHOR = {Casals, Roger and Gorsky, Eugene and Gorsky, Mikhail and
              Simental, Jos\'e},
     TITLE = {Algebraic weaves and braid varieties},
   JOURNAL = {Amer. J. Math.},
  FJOURNAL = {American Journal of Mathematics},
    VOLUME = {146},
      YEAR = {2024},
    NUMBER = {6},
     PAGES = {1469--1576},
      ISSN = {0002-9327,1080-6377},
   MRCLASS = {14R25 (57K10 57K33)},
  MRNUMBER = {4855860},
MRREVIEWER = {Priyadip\ Mondal},
       DOI = {10.1353/ajm.2024.a944357},
       URL = {https://doi.org/10.1353/ajm.2024.a944357},
}

@article {MR3176610,
    AUTHOR = {Knutson, Allen and Lam, Thomas and Speyer, David E},
     TITLE = {Projections of {R}ichardson varieties},
   JOURNAL = {J. Reine Angew. Math.},
  FJOURNAL = {Journal f\"ur die Reine und Angewandte Mathematik. [Crelle's
              Journal]},
    VOLUME = {687},
      YEAR = {2014},
     PAGES = {133--157},
      ISSN = {0075-4102,1435-5345},
   MRCLASS = {05E18 (14L35 14M15)},
  MRNUMBER = {3176610},
MRREVIEWER = {Anthony\ Henderson},
       DOI = {10.1515/crelle-2012-0045},
       URL = {https://doi.org/10.1515/crelle-2012-0045},
}

@article {MR1614555,
    AUTHOR = {Edidin, Dan and Graham, William},
     TITLE = {Equivariant intersection theory},
   JOURNAL = {Invent. Math.},
  FJOURNAL = {Inventiones Mathematicae},
    VOLUME = {131},
      YEAR = {1998},
    NUMBER = {3},
     PAGES = {595--634},
      ISSN = {0020-9910,1432-1297},
   MRCLASS = {14C17 (14F99)},
  MRNUMBER = {1614555},
MRREVIEWER = {Burt\ Totaro},
       DOI = {10.1007/s002220050214},
       URL = {https://doi.org/10.1007/s002220050214},
}

@article {MR2178703,
    AUTHOR = {Arapura, Donu},
     TITLE = {The {L}eray spectral sequence is motivic},
   JOURNAL = {Invent. Math.},
  FJOURNAL = {Inventiones Mathematicae},
    VOLUME = {160},
      YEAR = {2005},
    NUMBER = {3},
     PAGES = {567--589},
      ISSN = {0020-9910,1432-1297},
   MRCLASS = {14F42 (14D07)},
  MRNUMBER = {2178703},
MRREVIEWER = {Byungheup\ Jun},
       DOI = {10.1007/s00222-004-0416-x},
       URL = {https://doi.org/10.1007/s00222-004-0416-x},
}

@article {MR3512647,
    AUTHOR = {Escobar, Laura},
     TITLE = {Brick manifolds and toric varieties of brick polytopes},
   JOURNAL = {Electron. J. Combin.},
  FJOURNAL = {Electronic Journal of Combinatorics},
    VOLUME = {23},
      YEAR = {2016},
    NUMBER = {2},
     PAGES = {Paper 2.25, 18},
      ISSN = {1077-8926},
   MRCLASS = {14M15 (05E99 14M25)},
  MRNUMBER = {3512647},
MRREVIEWER = {Justin\ Brown},
       DOI = {10.37236/5038},
       URL = {https://doi.org/10.37236/5038},
}

@article{schwartz2026homfly,
  title={The {HOMFLY} polynomial of a forest quiver},
  author={Schwartz, Amanda},
  journal={International Mathematics Research Notices},
  volume={2026},
  number={8},
  pages={rnag069},
  year={2026},
  publisher={Oxford University Press}
}

@article{levinson2022class,
  title={Class groups of open {R}ichardson varieties in the {G}rassmannian are trivial},
  author={Levinson, Jake and Purbhoo, Kevin},
  journal={Journal of Commutative Algebra},
  volume={14},
  number={2},
  pages={267--275},
  year={2022},
  publisher={Rocky Mountain Mathematics Consortium Tempe, AZ, USA}
}

@article{galashin2024plabic,
  title={Plabic links, quivers, and skein relations},
  author={Galashin, Pavel and Lam, Thomas},
  journal={Algebraic Combinatorics},
  volume={7},
  number={2},
  pages={431--474},
  year={2024}
}

@misc{casals2025comparing,
  title={Comparing cluster algebras on braid varieties},
  author={Casals, Roger and Galashin, Pavel and Gorsky, Mikhail and Shen, Linhui and Sherman-Bennett, Melissa and Simental, Jos{\'e}},
  eprint = {2508.03816},
  archivePrefix = {arXiv},
  year={2025},
  note ={\textup{arXiv:}\texttt{2508.03816}}  
}

@misc{dudas2008note,
  title={Note on the {D}eodhar decomposition of a double {S}chubert cell},
  author={Dudas, Olivier},
  eprint={0807.2198},
  archivePrefix = {arXiv},
  year={2008},
  note ={\textup{arXiv:}\texttt{0807.2198}}
}

\end{document}